\documentclass[12pt,a4paper,reqno]{amsart} 

\title[Weyl product on quasi-Banach modulation spaces]{The Weyl product on quasi-Banach modulation spaces}

\author[Y. Chen]{Yuanyuan Chen}

\author[J. Toft]{Joachim Toft}

\author[P. Wahlberg]{Patrik Wahlberg}

\usepackage{amssymb}
\usepackage{latexsym}
\usepackage{amsmath}
\usepackage{mathrsfs}
\usepackage{euscript}
\usepackage{calc}                         %Detta beh�vs f�r att fixa
\usepackage{cite}
\usepackage{color}

\newcommand{\scal}[2]{\langle #1,#2\rangle}

\newcommand{\re}{\mathbf R}
\newcommand{\rr}[1]{\mathbf R^{#1}}
\newcommand{\zz}[1]{\mathbf Z^{#1}}
\newcommand{\nn}[1]{\mathbf N^{#1}}

\newcommand{\co}{\mathbf C}

\newcommand{\nm}[2]{\Vert #1\Vert _{#2}}

\newcommand{\op}{\operatorname{Op}}

\newcommand{\sets}[2]{\{ {\,}#1{\,};{\,}#2{\,}\} }
\newcommand{\Sets}[2]{\left \{ \, #1\, ;\, #2\, \right \} }
\newcommand{\ep}{\varepsilon}

\newcommand{\cdo}{\, \cdot \, }

\newcommand{\wpr}{{\text{\footnotesize{$\#$}}}}

\newcommand{\eabs}[1]{\langle #1\rangle}

\newcommand{\vrum}{\vspace{0.1cm}}

\newcommand{\masfR}{\mathsf R}

\newcommand{\maclS}{\mathcal S}
\newcommand{\maclV}{\mathcal V}

\newcommand{\mascF}{\mathscr F}
\newcommand{\mascS}{\mathscr S}
\newcommand{\mascP}{\mathscr P}
\newcommand{\splM}{\EuScript M}

\newcommand{\mabfa}{\boldsymbol a}
\newcommand{\mabfb}{\boldsymbol b}
\newcommand{\mabfc}{\boldsymbol c}
\newcommand{\mabfj}{\boldsymbol j}
\newcommand{\mabfk}{\boldsymbol k}

\newcommand{\mabfp}{{\boldsymbol p}}
\newcommand{\mabfq}{\boldsymbol q}

\newcommand{\M}{\operatorname{\mathbf{M}}}

\numberwithin{equation}{section}          %Detta g�r att man f�r
\newtheorem{thm}{Theorem}
\numberwithin{thm}{section}
\newtheorem*{tom}{\rubrik}
\newcommand{\rubrik}{}
\newtheorem{prop}[thm]{Proposition}

\newtheorem{lemma}[thm]{Lemma}

\theoremstyle{definition}

\newtheorem{defn}[thm]{Definition}

\theoremstyle{remark}

\newtheorem{rem}[thm]{Remark}

\begin{document}

\begin{abstract}
We study the bilinear Weyl product acting on quasi-Banach modulation spaces.
We find sufficient conditions for continuity of the Weyl product and we derive necessary conditions.  
The results extend known results for Banach modulation spaces. 
\end{abstract}

\keywords{Pseudodifferential calculus, Weyl product, modulation spaces, quasi-Banach spaces, Gabor frames}

\subjclass[2010]{35S05, 47G30, 42B35, 46A16, 46E35, 46F05}

\maketitle

{\let\thefootnote\relax\footnotetext{Corresponding author: Patrik Wahlberg,
phone +46470708637. Authors' addresses: 
Department of Mathematics, Linn\ae us University, 351 95 V\"axj\"o, Sweden. Email: yuanyuanchen0822@gmail.com, joachim.toft@lnu.se, patrik.wahlberg@lnu.se.}}

%%%%%%%%%%%%%%%%%%%%%%%%%%
\section{Introduction}
\label{sec0}
%%%%%%%%%%%%%%%%%%%%%%%%%%

In this paper we study the Weyl product acting on weighted modulation
spaces with Lebesgue parameters in $(0,\infty]$. 
We work out conditions on the weights and the Lebesgue parameters
that are sufficient for continuity of the Weyl product, and we also prove
necessary conditions. 

\par

The Weyl product or twisted product is the product of symbols in the Weyl
calculus of pseudodifferential operators corresponding to
operator composition. This means that the Weyl product 
\begin{equation*}
(a_1,a_2) \mapsto a_1 \wpr a_2
\end{equation*}
of two distributions $a_1,a_2$ defined on
the phase space $T^* \rr d\simeq \rr {2d}$ is defined by 
\begin{equation*}
\op^w (a_1 \wpr  a_2) = \op ^w (a_1) \circ \op^w (a_2) 
\end{equation*}
provided the composition is well defined. 

\par

Our result on sufficient conditions is as follows. Suppose $\omega_j$, $j=0,1,2$,
are moderate weights on $\rr {4d}$ that satisfy
\begin{equation*}
\omega_0(Z+X,Z-X) \lesssim \omega_1(Y+X,Y-X) \, \omega_2(Z+Y,Z-Y),
\quad X,Y,Z \in \rr {2d}. 
\end{equation*}
Suppose $p_j,q_j\in (0,\infty ]$, $j=0,1,2$, satisfy
\begin{equation*}
\frac 1{p_0} \le \frac 1{p_1}+\frac 1{p_2},
\end{equation*}
and either
\begin{equation*}
q_1,q_2 \le q_0 \le  \min (1,p_0)  
\end{equation*}
or
\begin{equation*}
\min (1,p_0) \le q_1,q_2 \le q_0
\quad \text{and} \quad
\frac 1{\min (1,p_0)}+\frac 1{q_0} \le \frac 1{q_1}+\frac 1{q_2}   .
\end{equation*}
Denote the Gelfand--Shilov space of order $1/2$ by $\maclS _{1/2}$, 
and the weighted modulation space with Lebesgue parameters $p,q>0$
and with weight $\omega$ by $\splM^{p,q}_{(\omega )}$. 
Then the map
$(a_1,a_2)\mapsto a_1 \wpr  a_2$ from
$\maclS _{1/2}(\rr {2d})\times \maclS _{1/2}(\rr {2d})$ to
$\maclS _{1/2}(\rr {2d})$ extends uniquely to a continuous map from
$\splM^{p_1,q_1}_{(\omega _1)}(\rr {2d})
\times \splM^{p_2,q_2}_{(\omega _2)}(\rr {2d})$ to
$\splM^{p_0,q_0}_{(\omega _0)}(\rr {2d})$, and
\begin{equation}\label{Wprodcont}
\nm{a_1 \wpr a_2}{\splM_{(\omega_0 )}^{p_0,q_0}}
\lesssim \nm{a_1}{\splM_{(\omega_1 )}^{p_1,q_1}}
\nm{a_2}{\splM_{(\omega_2 )}^{p_2,q_2}}. 
\end{equation}

\par

As a consequence for unweighted modulation spaces we obtain new
conditions on Lebesgue parameters that are sufficient for
$\splM^{p,q}(\rr {2d})$ to be an algebra: $q,p \in (0,\infty]$ and
$q \le \min(1,p)$. 

\par

The necessary conditions we deduce are as follows. 
Suppose \eqref{Wprodcont} holds for all $a_1,a_2 \in \mathscr S(\rr {2d})$,
for a triple of polynomial type weights $\omega_j$, $j=0,1,2$ interrelated in
a certain way, see \eqref{weightcond2}. Then 
\begin{equation*}
\frac 1{p_0}\le \frac 1{p_1}+\frac 1{p_2},\quad
\frac 1{p_0}\le \frac 1{q_1}+\frac 1{q_2}
\quad \text{and}\quad
q_1,q_2\le q_0,
\end{equation*}
which are strictly weaker than the sufficient conditions. 

\par

Our results for the Weyl product are special cases of results formulated
and proved for a family of pseudodifferential calculi parametrized by
real matrices $A \in \rr {d \times d}$. 
In fact we work with
a symbol product indexed by $A \in \rr {d \times d}$, denoted and defined by 
\begin{equation*}
\op _A(a \wpr _A b) = \op _A(a) \circ \op _A(b)
\end{equation*}
where $\op _A(a)$ is the $A$-indexed pseudodifferential operator with symbol $a$. 
This family of calculi contains the Weyl quantization as the special case $A=\frac12 I$. 

\par

The sufficient conditions and the necessary conditions that we find extend
results \cite{Cordero1,Holst1} where the same problem was studied for the
narrower range of Lebesgue parameters $[1,\infty]$. In the latter case
modulation spaces are Banach spaces, whereas they are merely
quasi-Banach spaces if a Lebesgue parameter is smaller than one. 

\par

The Weyl product on Banach modulation spaces has been studied in
 e.{\,}g. \cite{Cordero1,Holst1,Sjostrand1,Gc3,GrRz,Labate1,Toft1}.
In \cite{Cordero1} conditions on the Lebesgue parameters were found
that are both necessary and sufficient for continuity of the Weyl product,
thus characterizing the Weyl product acting on Banach modulation spaces. 

\par

One possible reason that we do not obtain characterizations in the full
range of Lebesgue parameters $(0,\infty]$ is that new difficulties arise
as soon as a Lebesgue parameter is smaller than one. The available
techniques are quite different, and many tools that are useful in the
Banach space case, e.g. duality and complex interpolation, are not
applicable or fraught with subtle difficulties. 

\par

Our technique to prove the sufficient conditions consists of a discretization
of the Weyl product by means of a Gabor frame. This reduces the
continuity of the Weyl product to the continuity of certain
infinite-dimensional matrix operators. A similar idea has
been developed in \cite{Wahlberg1}. 

\par

The paper is organized as follows. 
Section \ref{sec1} fixes notation and gives the background on
Gelfand--Shilov function and distribution spaces, pseudodifferential calculi,
modulation spaces, Gabor frames, and symbol product results for Banach
modulation spaces. 

\par

Section \ref{sec2} contains the result on sufficient conditions for continuity
on quasi-Banach modulation spaces (Theorem \ref{Thm:BilinCase}). 
Section \ref{sec3} contains the result on necessary conditions for continuity
on quasi-Banach modulation spaces (Theorem
\ref{Thm:SharpnessBiLinCase}). Finally in Appendix  we show a Fubini type
result for Gelfand--Shilov distributions that is needed in the definition of
the short-time Fourier transform of a Gelfand--Shilov distribution. 

\par

%%%%%%%%%%%%%%%%%%%%%%%%%%
\section{Preliminaries}
\label{sec1}
%%%%%%%%%%%%%%%%%%%%%%%%%%

\par

\subsection{Weight functions}

\par
 
 A \emph{weight} on $\rr d$ is a positive function $\omega
\in  L^\infty _{loc}(\rr d)$ such that $1/\omega \in  L^\infty _{loc}(\rr d)$.
We usually assume that $\omega$ is ($v$-)\emph{moderate},
for some positive function $v \in
 L^\infty _{loc}(\rr d)$. This means
\begin{equation}\label{moderate}
\omega (x+y) \lesssim \omega (x)v(y),\qquad x,y\in \rr d.
\end{equation}
Here $f(\theta ) \lesssim g(\theta )$ means that $f(\theta ) \le c g(\theta )$
holds uniformly for all $\theta$
in the intersection of the domains of $f$ and $g$
for some constant $c>0$, and we
write $f\asymp g$
when $f\lesssim g \lesssim f$. 
Note that \eqref{moderate} implies 
the estimates
\begin{equation}\label{moderateconseq}
v(-x)^{-1}\lesssim \omega (x)\lesssim v(x),\quad x\in \rr d.
\end{equation}
If $v$ in \eqref{moderate} can be chosen as a polynomial
then $\omega$ is called polynomially moderate or a weight of
polynomial type. We let
$\mascP (\rr d)$ and $\mascP _E(\rr d)$ be the sets of all weights of
polynomial type and moderate weights on $\rr d$, respectively.

\par

If $\omega \in \mascP _E(\rr d)$ then there exists $r>0$ such that 
$\omega$ is $v$-moderate for $v(x) = e^{r|x|}$\cite{Gc2.5}. 
Hence by \eqref{moderateconseq} for any $\omega \in \mascP
_E(\rr d)$ there is $r>0$ such that
\begin{equation}\label{WeightExpEst}
e^{-r|x|}\lesssim \omega (x)\lesssim e^{r|x|},\quad x\in \rr d.
\end{equation}

\par

A weight $v$ is called submultiplicative if $v$ is even and \eqref{moderate}
holds with $\omega =v$. In the paper $v$ and $v_j$ for
$j\ge 0$ will denote submultiplicative weights if not otherwise stated. 

\par

\subsection{Gelfand--Shilov spaces}

\par

Let $h,s\in \mathbf R_+$ be fixed. Then $\mathcal S_{s,h}(\rr d)$
is the set of all $f\in C^\infty (\rr d)$ such that
\begin{equation*}
\nm f{\mathcal S_{s,h}}\equiv \sup \frac {|x^\beta \partial ^\alpha
f(x)|}{h^{|\alpha | + |\beta |}(\alpha !\, \beta !)^s}
\end{equation*}
is finite, where the supremum is taken over all $\alpha ,\beta \in
\mathbf N^d$ and $x\in \rr d$.

\par

Obviously $\mathcal S_{s,h}$ is a Banach space which increases
with $h$ and $s$, and it is contained in the Schwartz space $\mathscr S$.
(Inclusions of function and distribution spaces understand embeddings.)
The topological dual $\mathcal S_{s,h}'(\rr d)$ of $\mathcal S_{s,h}(\rr d)$ is
a Banach space which contains $\mathscr S'(\rr d)$ (the tempered distributions).
If $s>1/2$, then $\mathcal
S_{s,h}$ and $\bigcup _{h>0}\maclS _{1/2,h}$ contain all finite linear
combinations of Hermite functions.

\par

The (Fourier invariant) \emph{Gelfand--Shilov spaces} $\mathcal S_{s}(\rr d)$ and
$\Sigma _s(\rr d)$ are the inductive and projective limits respectively
of $\mathcal S_{s,h}(\rr d)$ with respect to $h$. This implies
\begin{equation}\label{GSspacecond1}
\mathcal S_s(\rr d) = \bigcup _{h>0}\mathcal S_{s,h}(\rr d)
\quad \text{and}\quad \Sigma _{s}(\rr d) =\bigcap _{h>0}\mathcal
S_{s,h}(\rr d).
\end{equation}
The topology for $\mathcal S_s(\rr d)$ is the
strongest topology such that each inclusion $\mathcal S_{s,h}(\rr d) \subseteq\mathcal S_s(\rr d)$
is continuous. The projective limit $\Sigma _s(\rr d)$ is a Fr{\'e}chet
space with seminorms $\nm \cdot{\mathcal S_{s,h}}$, $h>0$.
It holds $\mathcal S _s(\rr d)\neq \{ 0\}$ if and only if
$s\ge 1/2$, and $\Sigma _s(\rr d)\neq \{ 0\}$
if and only if $s>1/2$.

\par

For every $\ep >0$ and $s>0$,
\begin{equation*}
\Sigma _s (\rr d)\subseteq \mathcal S_s(\rr d)\subseteq
\Sigma _{s+\ep}(\rr d).
\end{equation*}

\medspace

The \emph{Gelfand--Shilov distribution spaces} $\mathcal S_s'(\rr d)$
and $\Sigma _s'(\rr d)$ are the projective and inductive limits
respectively of $\mathcal S_{s,h}'(\rr d)$.  Hence
\begin{equation}\tag*{(\ref{GSspacecond1})$'$}
\mathcal S_s'(\rr d) = \bigcap _{h>0}\mathcal
S_{s,h}'(\rr d)\quad \text{and}\quad \Sigma _s'(\rr d)
=\bigcup _{h>0} \mathcal S_{s,h}'(\rr d).
\end{equation}
The space $\mathcal S_s'(\rr d)$ is the topological dual of $\mathcal
S_s(\rr d)$, and if $s>1/2$ then $\Sigma _s'
(\rr d)$ is the topological dual of $\Sigma _s(\rr d)$ \cite{GeSh}.

\par

The action of a distribution $f$ on a test function $\phi$ is written $\scal  f \phi$, 
and the conjugate linear action is written $(u,\phi) = \scal  {u} {\overline \phi}$, consistent with the $L^2$ inner product $(\cdo ,\cdo ) = (\cdo ,\cdo )_{L^2}$ which is conjugate linear in the second argument. 

\par

The Gelfand--Shilov (distribution) spaces enjoy many invariance properties, for instance under 
translation, dilation, tensorization, coordinate transformations and (partial) Fourier transformation.

\par

We use the normalization  
\begin{equation*}
 \mascF f (\xi )= \widehat f(\xi ) = (2\pi )^{-\frac d2} \int _{\rr
{d}} f(x)e^{-i\scal  x\xi }\, dx, \qquad \xi \in \rr d, 
\end{equation*}
of the Fourier transform 
of $f\in L^1(\rr d)$, where $\scal \cdo \cdo$ denotes the
scalar product on $\rr d$. The Fourier transform $\mascF$ extends 
uniquely to homeomorphisms on $\mathscr S'(\rr d)$, $\mathcal
S_s'(\rr d)$ and $\Sigma _s'(\rr d)$, and restricts to 
homeomorphisms on $\mathscr S(\rr d)$, $\mathcal S_s(\rr d)$
and $\Sigma _s(\rr d)$, and to a unitary operator on $L^2(\rr d)$.

The \emph{symplectic Fourier transform} of $a \in
\maclS _s (\rr {2d})$ where $s\ge 1/2$ is defined by 
\begin{equation*}
\mascF _\sigma a (X)
= \pi^{-d}\int _{\rr {2d}} a(Y) \, e^{2 i \sigma(X,Y)}\,  dY,
\end{equation*}
where $\sigma$ is the symplectic form
$$
\sigma(X,Y) = \scal y \xi -
\scal x \eta ,\qquad X=(x,\xi )\in \rr {2d},\ Y=(y,\eta )\in \rr {2d}.
$$
Since $\mascF _\sigma a(x,\xi ) = 2^d \mascF a(-2\xi ,2x)$,
the definition of $\mascF _\sigma$ extends in the same way as $\mascF$.

\medspace

Let $\phi \in \maclS _s (\rr d) \setminus \{0\}$. 
The \emph{short-time Fourier transform} (STFT) $V_\phi
f$ of $f\in \maclS _s'(\rr d)$ is the distribution on $\rr {2d}$ defined by 
\begin{equation}\label{defstft}
V_\phi f (x,\xi ) = \mathscr F(f\, \overline{\phi (\cdo -x)})(\xi ) =
(2\pi )^{-\frac d2} (f,\phi (\cdo -x) \, e^{i\scal \cdo \xi}).
\end{equation}
Note that $f\, \overline{\phi (\cdo -x)} \in \maclS _s'(\rr d)$ for fixed $x \in \rr d$,
and therefore its Fourier transform is an element in $\maclS _s'(\rr d)$. 
The fact that the Fourier transform is actually a smooth function given by the
formula \eqref{defstft} is proved in Appendix.  

\par

If $T(f,\phi )\equiv V_\phi f$ for $f,\phi \in \maclS _{1/2}(\rr d)$,
then $T$ extends uniquely to sequentially continuous mappings
\begin{alignat*}{2}
T\, &:\, & \maclS _s'(\rr d)\times \maclS _s(\rr d) &\to
\maclS _s '(\rr {2d})\bigcap C^\infty (\rr {2d}),
\\[1ex]
T\, &:\, & \maclS _s'(\rr d)\times \maclS _s'(\rr d) &\to
\maclS _s '(\rr {2d}),
\end{alignat*}
and similarly when $\maclS _s$ and $\maclS _s'$ are replaced
by $\Sigma _s$ and $\Sigma _s'$, respectively, or by
$\mascS$ and $\mascS '$, respectively \cite{CPRT10,Toft8}.

Similar properties hold true if instead $T(f,\phi ) = W_{f,\phi}$, where
$W_{f,\phi}$ is the cross-Wigner distribution of $f\in \maclS _s'(\rr d)$
and $\phi \in \maclS _s(\rr d)$, given by
$$
W_{f,\phi}(x,\xi ) \equiv \mascF (f(x+\cdo /2)\overline {\phi (x-\cdo /2)})(\xi ).
$$
If $q \in [1,\infty]$, $\omega \in \mascP _E(\rr d)$, $f\in L^q_{(\omega )}(\rr d)$ and $\phi \in \Sigma _1(\rr d)$ then  
$V_\phi f$ and $W_{f,\phi}$ take the forms
\begin{align}
V_\phi f(x,\xi ) &= (2\pi )^{-\frac d2}\int _{\rr d}f(y) \, \overline {\phi
(y-x)} \, e^{-i\scal y\xi}\, dy\tag*{(\ref{defstft})$'$}
\intertext{and}
W_{f,\phi} (x,\xi ) &= (2\pi )^{-\frac d2}\int _{\rr d}f(x+y/2) \, \overline {\phi
(x-y/2)} \, e^{-i\scal y\xi}\, dy.\notag
\end{align}
Here $L^p_{(\omega )}(\rr d)$ for $p\in (0,\infty ]$ and $\omega \in \mascP _E(\rr d)$
denotes the space of all $f\in L^p_{loc} (\rr d)$ such
that $f \omega \in L^p(\rr d)$, and $\nm f {L^p_{(\omega )}} = \nm {f \omega} {L^p}$. 

\par

For $a \in \maclS _{1/2} '(\rr {2d})$ and
$\Phi \in \maclS _{1/2} (\rr {2d})  \setminus 0$ the
\emph{symplectic STFT} $\maclV _{\Phi} a$
of $a$ with respect to $\Phi$ is defined similarly as the STFT by
\begin{equation}\nonumber
\maclV _{\Phi} a(X,Y) = \mascF _\sigma \big( a\, \overline{ \Phi (\cdo -X) }
\big) (Y),\quad X,Y \in \rr {2d}.
\end{equation}

\par

There are several ways to characterize Gelfand--Shilov function
and distribution spaces, for example in terms of expansions with
respect to Hermite functions \cite{GrPiRo,JaEi},  or in terms of the
Fourier transform and the STFT \cite{ChuChuKim,GrZi,Toft8, Toft15}.

\par

%%%
\subsection{An extended family of pseudodifferential calculi}
%%%

\par

We consider a family of pseudodifferential calculi parameterized
by the real $d\times d$ matrices, 
denoted $\M(d,\re)$ \cite{CaTo,To14}.
Let $s\ge 1/2$, let $a\in \maclS _s 
(\rr {2d})$ and let $A\in \M(d,\re)$ be fixed.
The pseudodifferential operator $\op _A(a)$ is the linear and
continuous operator
\begin{equation}\label{e0.5}
\op _A(a)f (x)
=
(2\pi  ) ^{-d}\iint_{\rr {2d}} a(x-A(x-y),\xi ) \, f(y) \, e^{i\scal {x-y}\xi }\, dyd\xi
\end{equation}
when $f\in \maclS _s(\rr d)$. For
 $a\in \maclS _s'(\rr {2d})$ the operator $\op _A(a)$ is defined as the linear and
continuous operator from $\maclS _s(\rr d)$ to $\maclS _s'(\rr d)$ with
distribution kernel
\begin{equation}\label{atkernel}
K_{a,A}(x,y)=(2\pi )^{-\frac d2} \mascF _2^{-1}a (x-A(x-y),x-y).
\end{equation}
Here $\mascF _2F$ is the partial Fourier transform of $F(x,y)\in
\maclS _s'(\rr {2d})$ with respect to the $y$ variable. This
definition makes sense since 
\begin{equation}\label{homeoF2tmap}
\mascF _2\quad \text{and}\quad F(x,y)\mapsto F(x-A(x-y),x-y)
\end{equation}
are homeomorphisms on $\maclS _s'(\rr {2d})$.

\par

An important special case is $A=t I$, with
$t\in \mathbf R$ and $I \in \M(d,\re)$ denoting
the identity matrix. In this case we write $\op _t(a) = \op _{t  I}(a)$. 
The normal or Kohn--Nirenberg representation $a(x,D)$
corresponds to $t=0$, and the Weyl quantization $\op ^w(a)$
corresponds to $t=\frac 12$. Thus
$$
a(x,D) = \op _0(a) = \op (a)
\quad \text{and}\quad \op ^w(a) = \op _{1/2}(a).
$$

The Weyl calculus is connected to the Wigner distribution with the formula
\begin{align*}
(\op^w (a) f,g)_{L^2(\rr d)} & = (2\pi )^{-\frac d2} (a,W_{g,f} )_{L^2(\rr {2d})},
\\[1ex]
& \qquad a \in  \maclS _{1/2} '(\rr {2d}), \quad f,g \in \maclS _{1/2} (\rr d). 
\end{align*}

\par

For every $a_1\in \maclS _s '(\rr {2d})$ and $A_1,A_2\in
\M(d,\re)$, there is a unique $a_2\in \maclS _s '(\rr {2d})$ such that
$\op _{A_1}(a_1) = \op _{A_2} (a_2)$. The following restatement of \cite[Proposition~1.1]{To14} explains the relations between $a_1$ and $a_2$. 

\par

\begin{prop}\label{Prop:CalculiTransfer}
Let $a_1,a_2\in \maclS _{1/2}'(\rr {2d})$ and $A_1,A_2\in \M(d,\mathbf R)$.
Then
\begin{equation}\label{calculitransform}
\op _{A_1}(a_1) = \op _{A_2}(a_2) \quad \Leftrightarrow \quad
e^{i\scal {A_2D_\xi}{D_x }}a_2(x,\xi )=e^{i\scal {A_1D_\xi}{D_x }}a_1(x,\xi ).
\end{equation}
\end{prop}

\par

\subsection{Modulation spaces}\label{subsec1.2}

\par

Let $\phi \in \maclS _{1/2}(\rr d)\setminus 0$,  $p,q\in (0,\infty ]$
and $\omega \in \mascP _E(\rr {2d})$.
The \emph{modulation space} $M^{p,q}_{(\omega )}(\rr d)$
is the set of all $f\in \maclS _{1/2}'(\rr d)$ such that $V_\phi f\in
L^{p,q}_{(\omega )}(\rr {2d})$, and $M^{p,q}_{(\omega )}(\rr d)$ is
equipped with the quasi-norm
\begin{equation}\label{modnorm2}
f\mapsto \nm f{M^{p,q}_{(\omega )}}\equiv \nm {V_\phi f}{L^{p,q}_{(\omega )}}.
\end{equation}

\par

On the even-dimensional phase space $\rr {2d}$ one may define
modulation spaces based on the symplectic STFT. 
Thus if $\omega \in \mascP _E (\rr {4d})$, $p,q\in (0,\infty ]$
and $\Phi \in \maclS _{1/2} (\rr {2d})\setminus 0$ are fixed, 
the \emph{symplectic modulation space}
$\splM ^{p,q}_{(\omega )}(\rr {2d})$ is obtained by replacing
the STFT $a\mapsto V_\Phi a$ by the 
symplectic STFT $a\mapsto \maclV _\Phi a$ in
\eqref{modnorm2}.
It holds (cf. \cite{Cordero1})
$$
\splM ^{p,q}_{(\omega )}(\rr {2d}) = M ^{p,q}_{(\omega _0)}(\rr {2d}),
\quad \omega(x,\xi, y, \eta) = \omega_0 (x,\xi, -2 \eta, 2 y)
$$
so all properties that are valid for $M ^{p,q}_{(\omega )}$
carry over to $\splM ^{p,q}_{(\omega )}$.

\par

In the following propositions we list some properties of modulation spaces
and refer to \cite{GaSa,Fei1,FG1,FG2,Gc2,Toft5} for proofs.

\par

\begin{prop}\label{p1.4A}
Let $p,q \in (0,\infty ]$. 
\begin{enumerate}
\item[{\rm{(1)}}] If $\omega \in \mascP  _{E}(\rr {2d})$ then $\Sigma _1(\rr d)
\subseteq M^{p,q}_{(\omega )}(\rr d) \subseteq \Sigma _1'(\rr d)$.

\vrum

\item[{\rm{(2)}}] If $\omega \in \mascP  _{E}(\rr {2d})$ satisfies \eqref{WeightExpEst}
for every $r >0$, then
$\maclS _1(\rr d)\subseteq M^{p,q}_{(\omega )}(\rr d)
\subseteq \maclS _1'(\rr d)$.

\vrum

\item[{\rm{(3)}}]  If $\omega \in \mascP (\rr {2d})$ then
$\mathscr S(\rr d)\subseteq M^{p,q}_{(\omega )}(\rr d)
\subseteq \mathscr S '(\rr d)$.
\end{enumerate}
\end{prop}

\par

\begin{prop}\label{p1.4B}
Let $r\in (0,1]$, $p,q,p_j,q_j\in (0,\infty ]$ and $\omega ,\omega _j,v\in
\mascP  _{E}(\rr {2d})$, $j=1,2$,
satisfy $r \le \min(p,q)$, $p _1\le p _2$, $q_1\le q_2$,  $\omega _2
\lesssim \omega _1$, and let $\omega$ be $v$-moderate. 
\begin{enumerate}
\item If $\phi \in M^r_{(v)}(\rr d)\setminus 0$ then
$f\in M^{p,q}_{(\omega )}(\rr d)$ if and only if
\eqref{modnorm2} is finite.
In particular $M^{p,q}_{(\omega )}(\rr d)$ is independent
of the choice of $\phi \in M^r_{(v)}(\rr d)\setminus 0$.
The space $M^{p,q}_{(\omega )}(\rr d)$ is a quasi-Banach
space under the quasi-norm \eqref{modnorm2}, and different
choices of $\phi$ give rise to equivalent quasi-norms.
If  $p,q\ge 1$ then $M^{p,q}_{(\omega )} (\rr d)$ is a Banach space with norm \eqref{modnorm2}. 

\vrum

\item[{\rm{(2)}}] $M^{p_1,q_1}_{(\omega _1)}(\rr d)\subseteq
M^{p_2,q_2}_{(\omega _2)}(\rr d)$.
\end{enumerate}
\end{prop}

\par

We will rely heavily on Gabor expansions
so we need the following concepts. 
The operators in Definition \ref{DefAnSynGabOps} are
well defined and continuous by the analysis in \cite[Chapters~11--14]{Gc2}. 

\par

\begin{defn}\label{DefAnSynGabOps}
Let $\Lambda \subseteq \rr d$ be a lattice,
let $\Lambda^2 = \Lambda \times \Lambda \subseteq \rr {2d}$, 
let $\omega ,v\in \mascP _E(\rr {2d})$ be such that
$\omega$ is $v$-moderate, and let $\phi ,\psi \in M^1_{(v)}(\rr d)$.
\begin{enumerate}
\item The \emph{Gabor analysis operator} $C_\phi = C^{\Lambda}_\phi$ is the operator from
$M^\infty _{(\omega )}(\rr d)$ to $\ell ^{\infty}_{(\omega)}(\Lambda^2 )$
given by
\begin{equation*}
C^\Lambda _\phi f \equiv \{ V_\phi f(j,\iota ) \} _{j,\iota \in \Lambda} \text ; 
\end{equation*}

\vrum

\item The \emph{Gabor synthesis operator} $D_\psi=D^{\Lambda}_\psi$ is the operator from
$\ell ^\infty _{(\omega )}(\Lambda^2 )$ to $M^\infty _{(\omega)}(\rr d)$
given by
\begin{equation*}
D^\Lambda _\psi c \equiv \sum _{j,\iota \in \Lambda} c(j,\iota)
\, e^{i\scal \cdo {\iota }} \psi (\cdo -j)\text ;
\end{equation*}

\vrum

\item The \emph{Gabor frame operator} $S_{\phi ,\psi}=S^{\Lambda}_{\phi ,\psi}$
is the operator on $M^\infty _{(\omega )}(\rr d)$
given by $D^\Lambda _\psi \circ
C^\Lambda _\phi$, i.{\,}e.
\begin{equation*}
S^{\Lambda}_{\phi ,\psi}f \equiv \sum _{j,\iota \in \Lambda} V_\phi f(j,\iota )
\, e^{i\scal \cdo {\iota }}\psi (\cdo -j).
\end{equation*}
\end{enumerate}
\end{defn}

\par

The following result is a consequence of \cite[Theorem~13.1.1]{Gc2} (see also 
\cite[Theorem~S]{Gc1}). 

\par

\begin{prop}\label{ThmS}
Suppose $v\in \mascP _E(\rr {2d})$ is submultiplicative, and let $\phi \in
M^1_{(v)}(\rr d)\setminus 0$.
There is a constant $\theta _0>0$ such that the Gabor frame operator
$S_{\phi ,\phi}^\Lambda$ is a homeomorphism on $M^1_{(v)}(\rr d)$
when $\Lambda =\theta \zz d$ and $\theta \in (0,\theta _0]$. 
The Gabor systems
\begin{equation}\label{DualFrames}
\{ e^{i\scal {\cdo }{\iota }}\phi (\cdo -j) \} _{(j,\iota )\in \Lambda}
\quad \text{and}\quad
\{ e^{i\scal {\cdo }{\iota }}\psi (\cdo -j) \} _{(j,\iota )\in \Lambda}
\end{equation}
are dual frames for $L^2(\rr d)$ when $\psi
= (S_{\phi ,\phi}^\Lambda )^{-1}\phi \in M^1_{(v)}(\rr d)$ and $\theta \in (0,\theta _0]$.
\end{prop}

\par

Let $v$, $\phi$ and $\Lambda$ be as in Proposition \ref{ThmS}. Then
$(S_{\phi ,\phi}^\Lambda )^{-1}\phi$
is called the \emph{canonical dual window of  $\phi$}, with respect to
$\Lambda$. We have
\begin{equation*}
S_{\phi ,\phi}^\Lambda (e^{i\scal \cdo {\iota }}f(\cdo -j)) =
e^{i\scal \cdo {\iota }}(S_{\phi ,\phi}^\Lambda f)(\cdo -j),
\end{equation*}
when $f\in M^\infty _{(1/v)}(\rr d)$ and $(j,\iota )\in \Lambda$.

\par

The next result concerns Gabor expansion of modulation spaces. 
It is a special case of \cite[Theorem 3.7]{Toft12} (see also
\cite[Corollaries~12.2.5 and 12.2.6]{Gc2} and \cite[Theorem~3.7]{GaSa}).

\par

\begin{prop}\label{ConseqThmS}
Let $\theta > 0$, $\Lambda = \theta \zz {d}$,
\begin{equation*}
\Lambda ^2=\Lambda \times \Lambda = \{ (j,\iota  )\} _{j,\iota \in \Lambda} \subseteq \rr {2d},
\end{equation*}
let $p,q, r  \in (0,\infty ]$ satisfy $r\le \min (1, p,q)$, and let
$\omega ,v\in \mascP _E(\rr {2d})$ be such that $\omega$ is
$v$-moderate. 
Suppose $\phi ,\psi \in M^r_{(v)}(\rr d)$ are such that \eqref{DualFrames}
are dual frames for $L^2(\rr d)$. Then the following is true:
\begin{enumerate}
\item 
The operators
\begin{equation*}
C^{\Lambda}_\phi: M^{p,q}_{(\omega)} (\rr d) \mapsto \ell^{p,q}_{(\omega)} (\Lambda^2)
\quad \text{and}\quad
D^{\Lambda}_\psi: \ell^{p,q}_{(\omega)} (\Lambda^2) \mapsto M^{p,q}_{(\omega)} (\rr d) 
\end{equation*}
are continuous.

\item The operators $S_{\phi ,\psi} \equiv D_\psi \circ C_\phi$ and
$S_{\psi ,\phi} \equiv D_\phi \circ C_\psi$ are both the identity map
on $M^{p,q} _{(\omega )}(\rr d)$, and if  $f\in M^{p,q}
_{(\omega )}(\rr d)$, then
\begin{equation}\label{GabExpForm}
\begin{aligned}
f & = \sum _{j,\iota \in \Lambda} V_\phi f (j,\iota )
\, e^{i\scal {\cdo }{\iota }}\psi (\cdo -j)
\\[1 ex]
& =
\sum _{j,\iota \in \Lambda }  V_\psi f (j,\iota )
\, e^{i\scal {\cdo }{\iota }}\phi (\cdo -j)
\end{aligned}
\end{equation}
with unconditional quasi-norm convergence in $M^{p,q} _{(\omega )}$
when $p,q <\infty$, and with convergence in
$M^\infty _{(\omega)}$ with respect to the weak$^*$ topology otherwise.

\item If  $f\in M^\infty _{(1/v)}(\rr d)$, then
\begin{equation*}
\displaystyle{\nm f{M^{p,q} _{(\omega )}}
\asymp
\nm {V_\phi f}
{\ell ^{p,q} _{(\omega )} (\Lambda ^2) }
\asymp
\nm {V_\psi f}{\ell ^{p,q} _{(\omega )} (\Lambda ^2)} }.
\end{equation*}
\end{enumerate}
\end{prop}

\par 

The series \eqref{GabExpForm}
are called \emph{Gabor expansions of $f$} with respect to $\phi$, $\psi$ and $\Lambda$.

\par

\begin{rem}\label{RemThmS}
There are many ways to achieve dual frames
\eqref{DualFrames} satisfying the required properties in
Proposition \ref{ConseqThmS}. In fact, let
$v,v_{0}\in \mascP _E(\rr {2d})$ be submultiplicative such that
$\omega $ is $v$-moderate and 
\begin{equation*}
L^1_{(v_0)}(\rr {2d}) \subseteq \bigcap_{0 < r \le 1} L^r(\rr {2d}). 
\end{equation*}
This inclusion is satisfied e.g. for $v_0 (x) = e^{\ep|x|}$ with $\ep>0$. 
Proposition \ref{ThmS} guarantees that
for some choice of $\phi ,\psi \in M^1_{(v_0v)}(\rr d)\subseteq
\bigcap_{0 < r \le 1} M^r_{(v)}(\rr d)$ and lattice $\Lambda \subseteq \rr d$, the sets in
\eqref{DualFrames} where $\psi = (S_{\phi ,\phi}^\Lambda )^{-1}\phi$, are dual frames.
\end{rem}

\par

We usually assume that $\Lambda =\theta \zz d$,
with $\theta >0$ small enough
to guarantee the hypotheses in Propositions \ref{ThmS} and \ref{ConseqThmS}
be fulfilled, and that the window function and its
dual belong to $M^r_{(v)}$ for every $r>0$. This
is always possible, in view of Remark \ref{RemThmS}.

\par

We need the following version of Proposition \ref{ThmS},
which is a consequence of \cite[Corollary~3.2]{CaTo} and the Fourier
invariance of $\Sigma _1(\rr {2d})$. 

\par

\begin{lemma}\label{aFrames}
Suppose $v\in \mascP _E(\rr {4d})$ is submultiplicative, let $\phi _1,\phi _2 \in \Sigma _1(\rr d)\setminus 0$,
and let
\begin{equation}\label{Eq:Rihaczek}
\Phi (x,\xi )=\phi _1(x) \, \overline {\widehat \phi _2(\xi )} \, e^{-i\scal x\xi }.
\end{equation}
Then there is a
lattice $\Lambda^2 \subseteq \rr {2d}$ such that
\begin{align}
&\{ \Phi (x-j,\xi -\iota )e^{i(\scal x\kappa +\scal k\xi )} \}
_{(j, \iota ),(k, \kappa ) \in \Lambda^2} \label{Eq:rihaczekframe}
\intertext{is a Gabor frame for $L^2(\rr {2d})$ with canonical dual frame}
&\{ \Psi (x-j,\xi -\iota )e^{i(\scal x\kappa +\scal k\xi )} \} _{(j,\iota ),(k,\kappa )
\in \Lambda^2}, \nonumber
\end{align}
and 
\begin{equation*}
\Psi = (S_{\Phi ,\Phi }^{\Lambda^2})^{-1}\Phi \in \bigcap_{r>0} M^r _{(v)}(\rr {2d}).
\end{equation*}
\end{lemma}

\par

The right-hand side of \eqref{Eq:Rihaczek} is called the cross-Rihaczek
distribution of $\phi _1$ and $\phi _2$ \cite{Gc2}.

\par

\begin{rem}
The last conclusion in Lemma \ref{aFrames} is a consequence of the sharper result 
\cite[Lemma~2]{Kostadinova1}.
\end{rem}

\par

\subsection{Pseudodifferential operators and Gabor analysis \cite{To13}}

\par

In order to discuss a reformulation of pseudodifferential operators by means of  Gabor analysis, we need the following matrix concepts.

\par

\begin{defn}\label{Def:matrix classes}
Let $p,q\in (0,\infty ]$, $\theta >0$, let $J$ be an index set and let
$\Lambda  =\theta \zz d$ be a lattice, and let $\omega \in \mascP _E(\rr {2d})$.
\begin{enumerate}
\item $\mathbb U _0'(J)$ is the set of all matrices
$A=(\mabfa (j,k))_{j,k\in J}$ with entries in $\co$;

\item $\mathbb U _0(J)$ is the set of all
$A=(\mabfa (j,k))_{j,k\in J}\in \mathbb U_0'(J)$ such that $\mabfa (j,k)
\neq 0$ for at most finitely many $(j,k)\in J\times J$;
\item if $A=(\mabfa (j,k))_{j,k\in \Lambda}\in \mathbb U_0'(\Lambda )$ then
\begin{equation}\label{Hhdef}
H_{A,\omega}(j,k) = \mabfa (j,j-k) \, \omega (j,j-k)
\quad \text{and}\quad
h_{A,p,\omega}(k) = \nm {H_{A,\omega}(\cdo ,k)}{\ell ^p}.
\end{equation}
The set $\mathbb U ^{p,q}(\omega ,\Lambda )$ consists of all matrices
$A=(\mabfa (j,k))_{j,k\in \Lambda}$ such that
\begin{equation}\label{Upqnorm}
\nm {(\mabfa (j,k))_{j,k\in \Lambda}}{\mathbb U ^{p,q}(\omega ,\Lambda )}
\equiv \nm {h_{A,p,\omega}}{\ell ^q}
\end{equation}
is finite.
\end{enumerate}
\end{defn}

\par

$\mathbb U ^{p,q}(\omega ,\Lambda )$ is a quasi-Banach space, and if $p,q\ge 1$
it is a Banach space.

\par

If $J$ is an index set then $A=(\mabfa (j,k))_{j,k\in J}
\in \mathbb U_0'(J)$ is called properly supported if the sets
$$
\sets {j\in J }{\mabfa (j,k_0) \neq 0}
\quad \text{and}\quad
\sets {k\in J }{\mabfa (j_0,k) \neq 0}
$$
are finite for every $j_0,k_0\in J$. 
The set of properly supported matrices is denoted $\mathbb U_{\operatorname p}(J)$, 
and evidently $\mathbb U_0(J) \subseteq \mathbb U_{\operatorname p}(J)$. 
The sets
$\mathbb U_0(J)$ and $\mathbb U_{\operatorname p}(J)$ are rings under
matrix multiplication, and
$\mathbb U_0'(J)$ is a $\mathbb U_{\operatorname p}(J)$-module
with respect to matrix multiplication.

\par

Let $\phi _1,\phi _2\in \Sigma _1(\rr d)\setminus 0$, let $\Phi$ be defined by
\eqref{Eq:Rihaczek}, 
let $\Lambda \subseteq \rr {d}$ be a lattice such that
$\Lambda ^2 = \Lambda \times \Lambda \subseteq \rr {2d}$
makes \eqref{Eq:rihaczekframe} a Gabor frame in accordance with Lemma \ref{aFrames}, 
and let $\Psi = (S_{\Phi ,\Phi }^{\Lambda^2})^{-1}\Phi$ be the canonical dual window of $\Phi$. 
Suppose $\omega _0\in \mascP _E(\rr {4d})$ and set
\begin{equation}\label{Eq:omega0omegaRel}
\omega (x,\xi ,y,\eta ) = \omega _0(x,\eta ,\xi -\eta ,y-x). 
\end{equation}
Let $a \in M^{p,q}_{(\omega _0)}(\rr {2d})$, define
\begin{multline}\label{Eq:SpecMatrixEl}
\mabfa (\mabfj ,\mabfk ) = V_\Psi a (j,\kappa ,\iota -\kappa, k -j) \, e^{i\scal {k-j}\kappa},
\\[1ex]
\text{where}\quad \mabfj = (j,\iota )\in \Lambda^2 \quad \text{and}\quad
\mabfk = (k,\kappa)\in \Lambda^2,
\end{multline}
and define the matrix 
\begin{equation*}
A = (\mabfa (\mabfj ,\mabfk ))_{\mabfj ,\mabfk \in \Lambda ^2}.
\end{equation*}
Then it follows from Propositions \ref{ThmS} and \ref{ConseqThmS} that
\begin{equation}\label{Eq:SymbMatrixEstLink}
\nm a{M^{p,q}_{(\omega _0)}} \asymp
\nm {A}
{\mathbb U ^{p,q}(\omega ,\Lambda ^2 )} 
\end{equation}
provided $\theta$ is sufficiently small. 

\par

By identifying matrices with corresponding linear operators,
\cite[Lemma~3.3]{To13} gives
\begin{equation}\label{Eq:OpMatrixLink}
\op (a) = D_{\phi _1}\circ A \circ C_{\phi _2}. 
\end{equation}
Hence, if $b \in \maclS _{1/2} (\rr {2d})$, 
\begin{equation*}
\mabfb (\mabfj ,\mabfk ) = V_\Psi b (j,\kappa ,\iota -\kappa, k -j) \,
e^{i\scal {k-j}\kappa}, \quad \mabfj, \mabfk \in \Lambda ^2,
\end{equation*} 
\begin{equation*}
B = (\mabfb (\mabfj ,\mabfk ))_{\mabfj ,\mabfk \in \Lambda ^2},
\end{equation*}
and the matrix $C$ is defined as
\begin{equation}\label{mellanmatris}
C = C_{\phi _2}\circ D_{\phi _1} 
\end{equation}
then
\begin{equation}\label{KNcomp}
\op (a\wpr _0b) = \op (a)\circ \op (b) = D_{\phi _1}\circ (A \circ C \circ B) \circ C_{\phi _2}
\end{equation}
and
\begin{equation}\label{KNprodnorm}
\nm {a\wpr _0b}{M^{p,q}_{(\omega _0)}} \asymp
\nm {A\circ C\circ B}{\mathbb U ^{p,q}(\omega ,\Lambda ^2 )}.
\end{equation}

\par

\subsection{Composition of pseudodifferential operators with
symbols in Banach modulation spaces}

\par

We recall algebraic results for pseudodifferential
operators with symbols in modulation spaces with Lebesgue
exponents not smaller than one \cite{Cordero1,Holst1,To14}.

\par

If $A\in \M(d,\mathbf R)$ then the
product $\wpr _A$ with $N$ factors
\begin{equation}\label{tProdmap}
(a_1,\dots ,a_N)\mapsto a_1\wpr _A \cdots \wpr _A a_N
\end{equation}
from $\maclS _{1/2}(\rr {2d})\times \cdots \times
\maclS _{1/2}(\rr {2d})$ to
$\maclS _{1/2}(\rr {2d})$ is defined by the formula
$$
\op _A(a _1 \wpr _A \cdots \wpr _A a_N) = \op _A(a_1)\circ \cdots \circ
\op _A(a_N).
$$
The map \eqref{tProdmap} can be extended in different ways,
e.{\,}g. as in \cite[Theorem 2.11]{Cordero1} which
is stated in a generalized form in Theorem \ref{thm0.3+0.4tOps} below.
Assume that the weight functions satisfy
\begin{multline}\label{weightcondtcalc}
\omega _0(T_A(X_N,X_0)) \lesssim  \prod _{j=1}^N
\omega _j(T_A(X_{j},X_{j-1})),
\quad X_0,\dots ,X_N \in \rr {2d},
\end{multline}
where
\begin{multline}\label{Ttdef}
T_A(X,Y) =(y+A(x-y),\xi +A^*(\eta -\xi ),\eta -\xi , x-y),
\\[1ex]
X=(x,\xi )\in \rr {2d},\ Y=(y,\eta )\in \rr {2d}.
\end{multline}
Here $A^*$ denotes $A$ transposed. 
Assume that the Lebesgue parameters satisfy
\begin{equation}\label{pqconditions}
\max \left ( \masfR _N(\mabfq ') ,0 \right )
\le  \min _{j=1,\dots ,N} \left ( \frac 1{p_0'},\frac 1{q_0},\frac 1{p_j},\frac 1{q_j'},
\masfR _N(\mabfp )\right )
\end{equation}
or
\begin{equation}\label{pqconditions3}
\masfR _N(\mabfp ) \ge 0,
\quad
\frac 1{q_0}\le  \frac 1{p_0'} \leq \frac12
\quad \text{and}\quad
\frac 1{q_j'}\le  \frac 1{p_j} \leq \frac12,
\quad
j=1,\dots ,N,
\end{equation}
where
\begin{align*}
\masfR _N(\mabfp ) &= ({N-1})^{-1}\left ( \sum _{j=1}^N\frac
1{p_j}-\frac 1{p_0}  \right ),
\\[1ex]
\mabfp &= (p_0,p_1,\dots ,p_N)\in [1,\infty ]^{N+1}.
\end{align*}

\par

\begin{thm}\label{thm0.3+0.4tOps}
Suppose $p_j,q_j\in [1,\infty ]$, $j=0,1,\dots , N$ satisfy
\eqref{pqconditions} or \eqref{pqconditions3}, and 
suppose $\omega _j \in \mascP _E(\rr {4d})$, $j=0,1,\dots ,N$,
satisfy \eqref{weightcondtcalc} and \eqref{Ttdef}. 
Then the map
\eqref{tProdmap} from $\maclS _{1/2}(\rr {2d}) \times \cdots \times
\maclS _{1/2}(\rr {2d})$ to $\maclS _{1/2}(\rr {2d})$
extends uniquely to a continuous and associative map from $M ^{p_1,q_1}
_{(\omega _1)}(\rr {2d}) \times \cdots \times M ^{p_N,q_N}
_{(\omega _N)}(\rr {2d})$ to $M ^{p_0,q_0} _{(\omega _0)}(\rr {2d})$, 
and
\begin{equation*}
\nm {a_1\wpr _A \cdots \wpr _A a_N}{M^{p_0,q_0}_{(\omega _0)}}
\lesssim \prod _{j=1}^N \nm {a_j}{M^{p_j,q_j}_{(\omega _j)}},
\end{equation*}
for $a_j\in M^{p_j,q_j}_{(\omega _j)}(\rr {2d})$,
$j=1,\dots ,N$.
\end{thm}

\par

Theorem \ref{thm0.3+0.4tOps} follows by similar arguments as
in the proof of \cite[Theorem 2.11]{Cordero1}. The details are left for the reader.

\par

\begin{rem}
We note that the definition of $T_A$ in \cite[Eq.~(2.30)]{Cordero1}
is incorrect and should be replaced by \eqref{Ttdef}
with $A=t I$, in order for \cite[Theorem~2.11]{Cordero1} to hold. A
corrected version of  \cite{Cordero1} has been posted on arxiv. 
\end{rem}

\par

%%%%%%%%%%%%%%%%%%%%%%%%%%%%%%%%%%%%%
\section{Composition of pseudodifferential operators with symbols in
quasi-Banach modulation spaces}\label{sec2}
%%%%%%%%%%%%%%%%%%%%%%%%%%%%%%%%%%%%%

\par

In this section we deduce a composition result for pseudodifferential operators
with symbols in modulation spaces with Lebesgue parameters in $(0,\infty]$. 

\par

If $A\in \M(d,\mathbf R)$ then the map
\begin{equation}\label{weylbilin}
(a_1,a_2)\mapsto a_1\wpr _A a_2
\end{equation}
from $\maclS _{1/2}(\rr {2d})\times
\maclS _{1/2}(\rr {2d})$ to
$\maclS _{1/2}(\rr {2d})$ is defined by 
\begin{equation*}
\op _A(a _1 \wpr _A a_2) = \op _A(a_1)\circ
\op _A(a_2). 
\end{equation*}

\par

The following result is the principal result of this paper.
It concerns sufficient conditions for the unique extension of
\eqref{weylbilin} to symbols in quasi-Banach modulation spaces. 

\par

The weight functions are assumed to obey the estimates 
\begin{multline}\label{Eq:WeightCondBilinCase}
\omega _0(T_A(Z,X)) \lesssim 
\omega _1(T_A(Y,X)) \, \omega _2(T_A(Z,Y)),
\quad X,Y ,Z \in \rr {2d},
\end{multline}
where
\begin{multline}\label{Ttdef2}
T_A(X,Y) =(y+A(x-y),\xi +A^*(\eta -\xi ), \eta-\xi , x-y),
\\[1ex]
X=(x,\xi )\in \rr {2d},\ Y=(y,\eta )\in \rr {2d}.
\end{multline}
(Cf. \eqref{weightcondtcalc} and \eqref{Ttdef}.)

\par

\begin{thm}\label{Thm:BilinCase}
Let $A\in \M(d,\mathbf R)$
and suppose $\omega _j \in \mascP _E(\rr {4d})$, $j=0,1,2$, satisfy
\eqref{Eq:WeightCondBilinCase} and \eqref{Ttdef2}. 
Suppose $p_j,q_j\in (0,\infty ]$, $j=0,1,2$, satisfy
\begin{equation}\label{Eq:CaseLesspjp}
\frac 1{p_0} \le \frac 1{p_1}+\frac 1{p_2}, 
\end{equation}
and either
\begin{equation}\label{EqCaseqLessp}
q_1,q_2 \le q_0 \le  \min (1,p_0)  
\end{equation}
or
\begin{equation}\label{EqCasepLessq}
\min (1,p_0) \le q_1,q_2 \le q_0
\quad \text{and} \quad
\frac 1{\min (1,p_0)}+\frac 1{q_0} \le \frac 1{q_1}+\frac 1{q_2} .
\end{equation}
Then the map
$(a_1,a_2)\mapsto a_1 \wpr _A a_2$ from
$\maclS _{1/2}(\rr {2d})\times \maclS _{1/2}(\rr {2d})$ to
$\maclS _{1/2}(\rr {2d})$ extends uniquely to a continuous map from
$M^{p_1,q_1}_{(\omega _1)}(\rr {2d})
\times M^{p_2,q_2}_{(\omega _2)}(\rr {2d})$ to
$M^{p_0,q_0}_{(\omega _0)}(\rr {2d})$, and
\begin{equation*}
\nm{a_1 \wpr _A a_2}{M_{(\omega_0 )}^{p_0,q_0}}
\lesssim \nm{a_1}{M_{(\omega_1 )}^{p_1,q_1}} \nm{a_2}{M_{(\omega_2 )}^{p_2,q_2}}
\end{equation*}
for all $a_1 \in M^{p_1,q_1}_{(\omega _1)}(\rr {2d})$ and $a_2 \in
M^{p_2,q_2}_{(\omega _2)}(\rr {2d})$.
\end{thm}

\par

We need some preparations for the proof. 
The following result contains the needed continuity properties for matrix operators.

\par

\begin{prop}\label{prop:TwoDimMatrixCase}
Let $\Lambda \subseteq \rr {d}$ be a lattice, let $p_j,q_j\in (0,\infty ]$, $j=0,1,2$,
be such that \eqref{Eq:CaseLesspjp} -- \eqref{EqCasepLessq} hold,
and suppose $\omega _0,\omega _1,\omega _2\in \mascP _E(\rr {2d})$
satisfy
\begin{equation*}
\omega _0(x,z)\lesssim \omega _1(x,y) \, \omega _2(y,z),\qquad x,y,z\in \rr d.
\end{equation*}
Then the map
$(A_1,A_2)\mapsto A_1 \circ A_2$ from $\mathbb U_0(\Lambda )
\times \mathbb U_0(\Lambda )$ to $\mathbb U_0(\Lambda )$
extends uniquely to a continuous map from
$\mathbb U^{p_1,q_1}(\omega _1,\Lambda )
\times \mathbb U^{p_2,q_2}(\omega _2,\Lambda )$ to
$\mathbb U^{p_0,q_0} (\omega _0,\Lambda )$,
and
\begin{equation}\label{Eq:MatrixCompEst}
\nm {A_1\circ A_2}{\mathbb U^{p_0,q_0}(\omega _0,\Lambda )} 
\lesssim \nm {A_1}{\mathbb U^{p_1,q_1}(\omega_1 ,\Lambda )}
\nm {A_2}{\mathbb U^{p_2,q_2}(\omega_2 ,\Lambda )}.
\end{equation}
\end{prop}

\par

\begin{proof}
Let $\mathbb U_{0,+}'(\Lambda )$ be the set of all $A\in
\mathbb U_0'(\Lambda )$ with non-negative entries, let $A_m= (\mabfa _m(j ,k ))
_{j,k \in \Lambda}\in \mathbb U_0(\Lambda )\bigcap \mathbb U_{0,+}'(\Lambda )$,
$m=1,2$, denote the matrix elements of \mbox{$B=A_1\circ A_2$} by $\mabfb (j ,k)$,
$j,k \in \Lambda$, and set $p=p_0$, $q=q_0$, $\omega =\omega _0$,
\begin{equation*}
{\mathfrak a}_m (j ,k )
\equiv |\mabfa _m(j ,j -k ) | \, \omega _m(j ,j -k )
\quad \text{and}\quad
{\mathfrak b}(j ,k ) \equiv |\mabfb (j ,j -k )| \, \omega (j ,j -k ),
\end{equation*}
$m=1,2$. Then 
\begin{align*}
\nm {A_m}{\mathbb U^{p_m,q_m}(\omega_m ,\Lambda )}
& = \nm { {\mathfrak a}_m}{\ell ^{p_m,q_m}}, \quad m=1,2,
\\[1ex]
\nm {A_1\circ A_2}{\mathbb U^{p,q}(\omega ,\Lambda )}
& = \nm { {\mathfrak b}}{\ell ^{p,q}},
\end{align*}
and we first prove
\begin{equation*}
\nm { {\mathfrak b}}{\ell ^{p,q}}\le \nm { {\mathfrak a}_1}{\ell ^{p_1,q_1}}
\nm { {\mathfrak a}_2}{\ell ^{p_2,q_2}}.
\end{equation*}

\par

We have
\begin{equation}\label{eq:buppsk}
{\mathfrak b}(j ,k) \le \sum _{l \in \Lambda } {\mathfrak a}_1 (j ,l )
{\mathfrak a}_2 (j -l ,k -l ). 
\end{equation}

\par

In order to estimate $\nm { {\mathfrak b}(\cdo ,k)}{\ell ^p}$ we consider
the cases $p<1$ and $p\ge 1$ separately.

\par

First assume that $p<1$, and set $r_j=\frac {p_j}p$. Then $\frac 1{r_1}+\frac 1{r_2}\ge 1$
by assumption \eqref{Eq:CaseLesspjp}, and
therefore H{\"o}lder's inequality yields for $k \in \Lambda$
\begin{align*} 
\nm { {\mathfrak b}(\cdo ,k)}{\ell ^p} ^p
& \le
\sum _{j \in \Lambda} \left ( \sum _{l \in \Lambda }
{\mathfrak a}_1 (j ,l )
\, {\mathfrak a}_2 (j -l ,k -l )
\right )^p
\\[1ex]
& \le
\sum _{l \in \Lambda } \sum _{j \in \Lambda } \left (
{\mathfrak a}_1 (j ,l )
\, {\mathfrak a}_2 (j -l ,k -l )
\right )^p
\\[1ex]
& \le 
\sum _{l \in \Lambda } \nm {{\mathfrak a}_1 (\cdo  ,l )^p}{\ell ^{r_1}}
\nm {{\mathfrak a}_2 (\cdo  ,k -l )^p}{\ell ^{r_2}}
\\[1ex]
& =
\sum _{l \in \Lambda } \nm {{\mathfrak a}_1 (\cdo  ,l )}{\ell ^{p_1}}^p
\nm {{\mathfrak a}_2 (\cdo  ,k -l )}{\ell ^{p_2}}^p,
\end{align*}
that is
\begin{equation*}
\nm { {\mathfrak b}(\cdo ,k)}{\ell ^p}
\le
(c_1*c_2 (k ))^{1/p},
\end{equation*}
with $c_m(k )=\nm { {\mathfrak a}_m(\cdo ,k )}{\ell ^{p_m}}^p$, $m=1,2$.

\par

In order to estimate $(c_1*c_2)^{1/p}$ we first assume
\eqref{EqCaseqLessp}. Then
\begin{multline*}
\nm { {\mathfrak b}}{\ell ^{p,q}}
\le \nm {(c_1*c_2)^{1/p}}{\ell ^q} = \nm {c_1*c_2}{\ell ^{q/p}}^{1/p}
\\[1ex]
\le (\nm {c_1}{\ell ^{q/p}} \nm {c_2}{\ell ^{q/p}} )^{1/p}
= \nm { {\mathfrak a}_1}{\ell ^{p_1,q}}\nm { {\mathfrak a}_2}{\ell ^{p_2,q}}
\le \nm { {\mathfrak a}_1}{\ell ^{p_1,q_1}}\nm { {\mathfrak a}_2}{\ell ^{p_2,q_2}},
\end{multline*}
and the result follows in this case.

\par

If instead \eqref{EqCasepLessq} holds then $q\ge q_1,q_2\ge p$,
and $r_j=q_j/p$, $j=1,2$, and $r=q/p$ satisfy
$$
r_1,r_2,r\ge 1
\quad \text{and}\quad
\frac 1{r_1}+\frac 1{r_2}\ge 1+\frac 1r.
$$
Hence Young's inequality may be applied and gives
\begin{multline*}
\nm { {\mathfrak b}}{\ell ^{p,q}}
\le \nm {c_1*c_2}{\ell ^r}^{1/p}
\le (\nm {c_1}{\ell ^{r_1}} \nm {c_2}{\ell ^{r_2}} )^{1/p}
= \nm { {\mathfrak a}_1}{\ell ^{p_1,q_1}}\nm { {\mathfrak a}_2}{\ell ^{p_2,q_2}},
\end{multline*}
and the result follows in this case as well.

\par

Next we consider the case $p\ge 1$. By Minkowski's and H{\"o}lder's 
inequalities and the assumption \eqref{Eq:CaseLesspjp} we get from 
\eqref{eq:buppsk} 
\begin{multline}\label{Eq:MaElEsts}
\nm { {\mathfrak b}(\cdo ,k)}{\ell ^p} 
\le
\sum _{l \in \Lambda } \left\| 
{\mathfrak a}_1 (\cdot ,l )
\, {\mathfrak a}_2 (\cdot -l ,k -l ) \right\|_{\ell ^p}
\\[1ex]
\le
\sum _{l \in \Lambda } \nm {{\mathfrak a}_1 (\cdo  ,l )}{\ell ^{p_1}}
\nm {{\mathfrak a}_2 (\cdo  ,k -l )}{\ell ^{p_2}}
 = c_1*c_2 (k ),
\end{multline}
where $c_m (k )=\nm { {\mathfrak a}_m(\cdo ,k )}{\ell ^{p_m}}$, $m=1,2$.

\par

If \eqref{EqCasepLessq} holds then $q \ge q_1,q_2 \ge 1$ and
Young's inequality gives
\begin{equation*}
\nm { {\mathfrak b}}{\ell ^{p,q}}
\le \nm {c_1*c_2}{\ell ^q}
\le \nm {c_1}{\ell ^{q_1}} \nm {c_2}{\ell ^{q_2}} 
= \nm { {\mathfrak a}_1}{\ell ^{p_1,q_1}}\nm { {\mathfrak a}_2}{\ell ^{p_2,q_2}}
\end{equation*}
and the result follows. If instead \eqref{EqCaseqLessp} holds then 
$q \le 1$ and \eqref{Eq:MaElEsts} gives
\begin{multline*}
\nm { {\mathfrak b}}{\ell ^{p,q}}
\le \nm {c_1*c_2}{\ell ^q}
\le \nm {c_1}{\ell ^{q}} \nm {c_2}{\ell ^{q}} 
= \nm { {\mathfrak a}_1}{\ell ^{p_1,q}}\nm { {\mathfrak a}_2}{\ell ^{p_2,q}}
\\[1ex]
\le \nm { {\mathfrak a}_1}{\ell ^{p_1,q_1}}\nm { {\mathfrak a}_2}{\ell ^{p_2,q_2}}.
\end{multline*}
Thus we have proved \eqref{Eq:MatrixCompEst} when
$A_1,A_2\in \mathbb U_0(\Lambda )\bigcap \mathbb U_{0,+}'(\Lambda )$.

\par

By Beppo--Levi's theorem or
Fatou's lemma applied to the previous situation we obtain that $A_1\circ A_2$
is uniquely defined as an element in $\mathbb U^{p,q}
(\omega ,\Lambda)$ and \eqref{Eq:MatrixCompEst}
holds, provided $A_m\in \mathbb U^{p_m,q_m}(\omega _m,\Lambda )
\bigcap \mathbb U_{0,+}'(\Lambda )$ for $m=1,2$.

\par 

For $A_m\in \mathbb U^{p_m,q_m}(\omega _m,\Lambda )$,
$m=1,2$, there are unique 
$$
A_{m,k}\in \mathbb U^{p_m,q_m}(\omega _m,\Lambda )\bigcap \mathbb U_{0,+}'(\Lambda ), \quad m=1,2, \quad k=1,\dots ,4, 
$$ 
such that
$$
A_m=\sum _{k=1}^4i^kA_{m,k},
$$
and we have
$$
\nm {A_{m,k}}{\mathbb U^{p_m,q_m}(\omega _m,\Lambda )}
\le
\nm {A_m}{\mathbb U^{p_m,q_m}(\omega _m,\Lambda )}, \quad m=1,2, \quad k=1,\dots,4.
$$
Since the assertion holds true for $A_{1,k}$ and $A_{2,l}$ in place
of $A_1$ and $A_2$, it follows from the latter estimate that
$$
A_1\circ A_2 = \sum _{k,l=1}^4i^{k+l}A_{1,k}\circ A_{2,l}\in
\mathbb U^{p,q}(\omega ,\Lambda )
$$
is uniquely defined and that \eqref{Eq:MatrixCompEst}
holds for $A_m\in \mathbb U^{p_m,q_m}(\omega _m,\Lambda )$, $m=1,2$. 
\end{proof}

\par

We also need the
following result on the composition of the analysis operator and the
synthesis operator defined by two Gabor systems. 

\par

\begin{lemma}\label{lem:STFTMatrix}
Suppose $\Lambda \subseteq \rr d$ is a lattice, $\Lambda ^2=\Lambda \times \Lambda $
and $\phi _1,\phi _2\in \Sigma _1(\rr d) \setminus 0$. 
Let $C_{\phi _2} = C_{\phi _2}^\Lambda$ be the Gabor analysis operator and 
let $D_{\phi _1} = D_{\phi _1}^\Lambda$ be the Gabor synthesis operator 
defined by $\phi_2$ and $\phi_1$ respectively, and $\Lambda$. 
Then $C_{\phi _2}\circ D_{\phi _1}$ is the matrix
$(\mabfc (\mabfj , \mabfk ))_{\mabfj ,\mabfk \in \Lambda ^2}$ where
\begin{equation}\label{Eq:Celem}
\mabfc  (\mabfj , \mabfk ) = e^{i\scal k{\kappa -\iota}} V_{\phi _2}\phi _1(\mabfj -\mabfk ),
\quad \mabfj = (j,\iota ),\ \mabfk = (k,\kappa ).
\end{equation}
If $\omega_0 (X,Y)= \omega (X-Y)$,
$X, Y\in \rr {2d}$ for $\omega \in \mascP _E(\rr {2d})$, then
\begin{equation}\label{Eq:MatrixCincl}
(\mabfc (\mabfj , \mabfk ))_{\mabfj ,\mabfk \in \Lambda ^2}
\in \bigcap_{\stackrel{q>0}{\omega \in \mascP _E(\rr {2d})}} \mathbb U^{\infty ,q}(  \omega_0 ,\Lambda ^2). 
\end{equation}
\end{lemma}

\par

\begin{proof}
Let $f$ be a sequence on $\Lambda ^2$ such that $f(\mabfk )\neq 0$ for at most a finite number of
$\mabfk \in \Lambda ^2$. Then
\begin{equation*}
D_{\phi _1}f = \sum _{\mabfk \in \Lambda ^2} f(\mabfk )\phi _{1,\mabfk} ,\qquad \phi _{1,\mabfk}\equiv
\phi _1(\cdo -k)e^{i\scal \cdo \kappa},\ \mabfk = (k,\kappa),
\end{equation*}
and
\begin{equation*}
C_{\phi _2} (D_{\phi _1}f)(\mabfj ) = V_{\phi _2} (D_{\phi _1}f)(\mabfj )
= \sum _{\mabfk \in \Lambda ^2} V_{\phi _2} \phi _{1,\mabfk} (\mabfj) f(\mabfk ).
\end{equation*}
If $\mabfj = (j,\iota )$ then $C_{\phi _2} \circ D_{\phi _1}$
is hence given by the matrix $C = (\mabfc (\mabfj ,\mabfk ))_{\mabfj,\mabfk
\in \Lambda ^2}$ where
\begin{align*}
\mabfc(\mabfj ,\mabfk ) = V_{\phi _2} \phi _{1,\mabfk} (\mabfj)
& =
(2\pi )^{-d/2}\int_{\rr d} \phi _{1,\mabfk}(y) \, \overline {\phi _2(y-j)} \, e^{-i\scal y\iota}\, dy
\\[1ex]
& =
(2\pi )^{-d/2}e^{i\scal k{\kappa - \iota}}
\int_{\rr d} \phi _1(y) \, \overline {\phi _2(y-(j-k))} \, e^{-i\scal y{\iota -\kappa}}\, dy
\\[1ex]
& =
e^{i\scal k{\kappa - \iota}}V_{\phi _2}\phi _1(\mabfj - \mabfk )
\end{align*}
which proves \eqref{Eq:Celem}.

\par

It remains to prove \eqref{Eq:MatrixCincl}. Let $\omega \in \mascP _E(\rr {2d})$ and $q>0$.
Since $\phi _1,\phi _2 \in \Sigma _1(\rr d)$ we have by \cite[Theorem~2.4]{Toft8}
\begin{equation*}
| V_{\phi _2}\phi _1 (x,\xi )|\lesssim e^{-r(|x|+|\xi |)}
\end{equation*}
for every $r>0$. From \eqref{Eq:Celem} and \eqref{Hhdef} we obtain
\begin{equation*}
h_{C,\infty ,\omega_0 } (\mabfk) 
= \sup_{\mabfj \in \Lambda^2} |H_{C,\omega_0 }(\mabfj ,\mabfk)|
= | V_{\phi _2}\phi _1 (\mabfk ) \, \omega (\mabfk )| .
\end{equation*}
A combination of these relations and \eqref{Upqnorm} now give
\begin{equation*}
\nm C{\mathbb U^{\infty ,q}(  \omega_0,\Lambda ^2)}
= \nm{ h_{C,\infty ,\omega_0 } }{\ell^q}
= \nm {V_{\phi _2}\phi _1\cdot \omega}{\ell ^q(\Lambda^2 )}
<\infty .
\end{equation*}
Hence $C\in \mathbb U^{\infty ,q}(  \omega_0,\Lambda ^2)$ for any $\omega \in \mascP _E(\rr {2d})$ and any $q>0$. 
\end{proof}

\par

\begin{proof}[Proof of Theorem \ref{Thm:BilinCase}]
By \cite[Proposition~2.8]{To14} and Proposition \ref{Prop:CalculiTransfer}
we may assume that $A=0$. 
Pick $\phi_1, \phi_2 \in \Sigma_1(\rr d) \setminus 0$ and 
a lattice $\Lambda \subseteq \rr d$ such that $\Phi, \Psi$ and $\Lambda^2 = \Lambda \times
\Lambda \subseteq \rr {2d}$ are as in Lemma \ref{aFrames}.
Let finally $a_m\in M^{p_m,q_m}_{(\omega _m)}(\rr {2d})$, $m=1,2$. 

\par

By \eqref{Eq:omega0omegaRel} -- \eqref{Eq:OpMatrixLink} we have for $m=1,2$
\begin{align}
\nm {a_m}{M^{p_m,q_m}_{(\omega _m)}}
&\asymp
\nm {A_m}{\mathbb U^{p_m,q_m}(\vartheta _m,\Lambda ^2)}
\label{Eq:ModNormDiscNormEquiv}
\intertext{and}
\op (a_m) &= D_{\phi _1}\circ A_m \circ C_{\phi _2}
\intertext{where}
A_m &= (\mabfa _m(\mabfj ,\mabfk ))_{\mabfj ,\mabfk \in \Lambda ^2},
\notag
\end{align}
\begin{equation*}
\mabfa _m(\mabfj ,\mabfk ) \equiv e^{i\scal {k-j}\kappa} V_\Psi a_m (j,\kappa
,\iota -\kappa ,k-j), \quad \mabfj = (j,\iota ), \ \mabfk = (k,\kappa) \in \Lambda^2, 
\end{equation*}
and
\begin{equation*}
\vartheta _m(x,\xi ,y,\eta ) = \omega _m(x,\eta ,\xi -\eta ,y-x ).
\end{equation*}
Condition \eqref{Eq:WeightCondBilinCase} means for the weights $\vartheta _m$, $m=0,1,2$,
\begin{equation}\label{varthetaweight}
\vartheta _0(X,Y) \lesssim \vartheta _1(X ,Z )
\, \vartheta _2(Z ,Y),\qquad X,Y,Z \in \rr {2d}.
\end{equation}

\par

Pick $v_1 \in \mascP _E(\rr d)$ even so that $\omega _2$
is $v_2$-moderate with 
\begin{equation*}
v_2 =v_1 \otimes v_1 \otimes v_1 \otimes v_1 \in \mascP _E(\rr {4d}),
\end{equation*}
set $v = v_1^2 \otimes v_1 \in \mascP _E(\rr {2d})$ and 
\begin{equation*}
v_0(X,Y)=v(X-Y)\in \mascP _E(\rr {4d}), \quad X,Y \in \rr {2d}. 
\end{equation*}
Then $v_0$ is designed to guarantee 
\begin{equation}\label{matrixweight1}
\vartheta _2(X,Y) \lesssim v_0(X,Z) \, \vartheta _2(Z,Y), \qquad X,Y,Z \in \rr {2d}. 
\end{equation}

\par

We have by \eqref{mellanmatris} and \eqref{KNcomp}
\begin{equation*}
\op (a_1)\circ \op (a_2) = D_{\phi _1}\circ A \circ C_{\phi _2},
\end{equation*}
where
\begin{equation*}
A = A_1\circ C \circ A_2
\end{equation*}
and $C=C_{\phi _2}\circ D_{\phi _1}$.
By Lemma \ref{lem:STFTMatrix}
\begin{equation*}
C\in \underset {r>0}{\textstyle{\bigcap}}\mathbb U^{\infty ,r}(v_0,\Lambda ^2).
\end{equation*}
Set $r=\min(1,p_2,q_2)$. 
Then we obtain from \eqref{KNprodnorm}, \eqref{varthetaweight},
\eqref{matrixweight1} and Proposition \ref{prop:TwoDimMatrixCase} applied twice
\begin{align*}
\nm {a_1 \wpr _0 a_2}{M^{p_0,q_0}_{(\omega _0)}} 
& \asymp
\nm {A_1 \circ C \circ A_2}{\mathbb U ^{p_0,q_0}(\vartheta_0 ,\Lambda ^2 )}
\\[1ex]
& \lesssim \nm {A_1}{\mathbb U ^{p_1,q_1}(\vartheta_1 ,\Lambda ^2 )}
\nm {C \circ A_2}{\mathbb U ^{p_2,q_2}(\vartheta_2 ,\Lambda ^2 )}
\\[1ex]
& \lesssim \nm {A_1}{\mathbb U ^{p_1,q_1}(\vartheta_1 ,\Lambda ^2 )} 
\nm {C}{\mathbb U ^{\infty,r}(v_0 ,\Lambda ^2 )}
\nm {A_2}{\mathbb U ^{p_2,q_2}(\vartheta_2 ,\Lambda ^2 )}
\\[1ex]
& \asymp \nm {A_1}{\mathbb U ^{p_1,q_1}(\vartheta_1 ,\Lambda ^2 )}
\nm {A_2}{\mathbb U ^{p_2,q_2}(\vartheta_2 ,\Lambda ^2 )}
\\[1ex]
& \asymp \nm {a_1}{M^{p_1,q_1}_{(\omega _1)}} \nm {a_2}{M^{p_2,q_2}_{(\omega _2)}}.
\end{align*}

It remains to prove the claimed uniqueness of the extension. 
If \eqref{EqCaseqLessp} holds then $M^{p_j,q_j}_{(\omega _j)}
\subseteq M^{\infty,1}_{(\omega _j)}$, $j=1,2$, and
$M^{p_0,q_0}_{(\omega _0)} \subseteq M^{\infty,1}_{(\omega _0)}$. 
Then the claim follows from the uniqueness of the extension 
\begin{equation}\label{allqareone}
M^{\infty,1}_{(\omega _1)} \wpr_A M^{\infty,1}_{(\omega _2)}
\subseteq M^{\infty,1}_{(\omega _0)} 
\end{equation}
which is proved in \cite[Theorem~2.11]{Cordero1}. 

\par

Suppose \eqref{EqCasepLessq} holds. 
Then the same argument applies if $q \le 1$, and if $p \ge 1$ then 
the claim is a consequence of the uniqueness of the extension
\begin{equation}\label{youngq}
M^{\infty,q_1}_{(\omega _1)} \wpr_A M^{\infty,q_2}_{(\omega_2)}
\subseteq M^{\infty,q}_{(\omega _0)} 
\end{equation}
which is again proved in \cite[Theorem~2.11]{Cordero1}. 
Suppose $p<1<q$. 
If $q_1, q_2 \ge 1$ then the uniqueness follows again from the
uniqueness of \eqref{youngq}. If $q_1 \ge 1 > q_2$ then it follows
from the uniqueness of \eqref{youngq} with $q_2$ replaced by $1$, 
and analogously for $q_2 \ge 1 > q_1$. 
Finally if $q_1,q_2<1$ then the uniqueness follows from the
uniqueness of \eqref{allqareone}.
\end{proof}

\par

Let $p,q \in (0,\infty]$ and set $r=\min(1,p,q)$. 
A particular case of Theorem \ref{Thm:BilinCase} is the inclusion
\begin{equation*}
M^{p,q}_{(\omega _0)} \wpr_A M^{\infty,r}_{(\omega_2)} \subseteq M^{p,q}_{(\omega _0)} 
\end{equation*}
where the weights $\omega_0, \omega_2 \in \mascP _E(\rr {4d})$ satisfy
\begin{equation*}
\omega _0(T_A(Z,X)) \lesssim 
\omega _0(T_A(Y,X)) \, \omega _2(T_A(Z,Y)),
\quad X,Y ,Z \in \rr {2d},
\end{equation*}
and $T_A$ is defined by \eqref{Ttdef2}. 

\par

We also note that $M^{p,q}_{(\omega)}$ is an algebra under the product $\wpr_A$ provided 
$p,q \in (0,\infty]$, $q \le \min(1,p)$ and $\omega \in \mascP _E(\rr {4d})$ satisfies 
\begin{equation*}
\omega (T_A(Z,X)) \lesssim 
\omega (T_A(Y,X)) \, \omega (T_A(Z,Y)),
\quad X,Y ,Z \in \rr {2d},
\end{equation*}

\par

%%%%%%%%%%%%%%%%%%%%%%%%%%
\section{Necessary conditions}
\label{sec3}
%%%%%%%%%%%%%%%%%%%%%%%%%%

\par

In this final section we show that some of the sufficient conditions
in Theorem \ref{Thm:BilinCase} are necessary. 
We need the following lemma that concerns Wigner distributions. 

\par

\begin{lemma}\label{lemFseries}
Let $q_0, q \in (0,\infty ]$ satisfy $q_0 < q$,
let
$$
\phi (x)=\pi ^{-\frac d4}e^{-\frac {|x|^2}2}
\quad \text{for}\quad
x \in \rr d,
$$ 
let $\Lambda \subseteq \rr d$ be a lattice,
let $\mathbf c =\{ c(\kappa )\} _{\kappa \in \Lambda}\in \ell ^{q}(\Lambda )
\setminus \ell ^{q_0}(\Lambda )$,
where $c(\kappa )\ge 0$ for all $\kappa \in \Lambda$, 
and finally let
\begin{equation*}
f(x) = \sum _{\kappa \in \Lambda}c(\kappa )e^{i\scal x\kappa}  \phi (x)  \in \mascS'(\rr d). 
\end{equation*}
Then
\begin{equation}\label{Eq:fWignerfMod1}
f\in \underset {p>0}{\textstyle{\bigcap}} M^{p,q}(\rr d) \setminus M^{\infty ,q_0}(\rr d)
\end{equation}
and 
\begin{equation}\label{Eq:fWignerfMod2}
W_{f,\phi}\in  \underset {p>0}{\textstyle{\bigcap}} \splM^{p,q}(\rr {2d}) .
\end{equation}
\end{lemma}

\par

\begin{proof}
By replacing $\Lambda$ by a sufficiently dense lattice $\Lambda _0$,
containing $\Lambda$ and letting $c(\kappa )=0$ when
$\kappa \in \Lambda _0\setminus \Lambda$, we reduce ourselves to a situation
where the hypothesis in Proposition \ref{ConseqThmS} is
fulfilled.
Hence we may assume that \eqref{DualFrames} are dual frames
for $L^2(\rr d)$.

\par

First we show \eqref{Eq:fWignerfMod1}. 
(Cf. \cite[Proposition~2.6]{Reich1}.)
On one hand we have $\nm f {M^{p,q}} \lesssim \nm {\mathbf c} {\ell^q}$
for any $p>0$ due to Proposition \ref{ConseqThmS} (1). 
Thus $f \in \underset {p>0}{\textstyle{\bigcap}} M^{p,q}(\rr d)$. 
On the other hand $f \notin M^{\infty ,q_0}(\rr d)$. 

\par

In fact, set $\phi_1(x) = (2 \pi)^{-d/2} e^{-\frac 14| x |^2}$ for
$x \in \rr d$.  Since
\begin{align*}
V_\phi f(0,\iota) = (2 \pi)^{-\frac d2} \sum _{\kappa \in \Lambda} c(\iota-\kappa)
\, e^{-\frac{1}{4}| \kappa |^2} = \mathbf c * \phi_1 (\iota )
\end{align*}
we obtain
\begin{multline*}
\nm {\mathbf c} {\ell^{q_0}}^{q_0}
= \sum_{\iota \in \Lambda} c(\iota)^{q_0}
\le (2 \pi)^{\frac{d}{2q_0}} \sum_{\iota \in \Lambda}
(\mathbf c * \phi_1(\iota ) )^{q_0}
\\[1ex]
= (2 \pi)^{\frac{d}{2q_0}}  \sum_{\iota \in \Lambda} \left| V_\phi f(0,\iota)  \right|^{q_0}
\le (2 \pi)^{\frac{d}{2q_0}}  \sum_{\iota \in \Lambda} \left( \sup_{j \in \Lambda}
\left| V_\phi f(j,\iota) \right| \right)^{q_0}
\\[1ex]
= (2 \pi)^{\frac{d}{2q_0}} \nm {V_\phi f} {\ell^{\infty,q_0}(\Lambda^2)}^{q_0}
\asymp \nm f {M^{\infty,q_0}}^{q_0}
\end{multline*}
again by Proposition \ref{ConseqThmS}. 
Thus it must hold $f \notin M^{\infty ,q_0}(\rr d)$, since otherwise we get the
contradiction $\mathbf c \in \ell ^{q_0}(\Lambda )$. 
We have now showed \eqref{Eq:fWignerfMod1}. 

\par

In order to prove \eqref{Eq:fWignerfMod2}, set $a=W_{f,\phi} \in \mascS'(\rr {2d})$. 
Since $\splM^{p,q}$
is increasing with respect to $p$ and $q$, it suffices to intersect in
\eqref{Eq:fWignerfMod2} over $0 < p \le \min(1,q)$. 
We have
\begin{equation*}
\nm a{\splM^{p,q}} \asymp \nm {V_{\Phi}W_{f,\phi}}{\ell ^{p,q}(\Lambda ^4)},
\end{equation*}
where $\Phi (x,\xi )=(2\pi )^{-\frac d2}e^{-(|x|^2+|\xi |^2)}$,
and 
\begin{equation*}
\Lambda ^4=\Lambda \times \Lambda \times \Lambda \times \Lambda \subseteq \rr {4d}. 
\end{equation*}

\par

By straightforward computations we get
\begin{multline*}
a(x,\xi ) = W_{f,\phi} (x,\xi )
\\[1ex]
= (2\pi )^{-\frac d2} \sum _{\kappa \in \Lambda} c (\kappa) \,  \pi ^{-\frac d2}
\int_{\rr d} e^{-\frac 12(|x-\frac y2|^2+|x+\frac y2|^2)}
e^{i(\scal x \kappa -\scal{y}{\xi-\frac \kappa 2})} \, dy
\\[1ex]
= 2^{\frac d2} \pi ^{-\frac d2} \sum _{\kappa \in \Lambda} c (\kappa)
\, e^{-|x|^2-|\xi -\frac \kappa 2|^2}e^{i\scal x\kappa} .
\end{multline*}
This gives
\begin{equation*}
V_\Phi a (x,\xi ,\eta ,y) = 2^{\frac d2} \pi ^{-\frac d2} 
\sum_{\kappa \in \Lambda} c (\kappa) F_\kappa (x,\xi ,\eta ,y)
\end{equation*}
where
\begin{multline*}
F_\kappa (x,\xi ,\eta ,y)
\\[1ex]
= (2\pi) ^{-\frac {3d}2}
\iint_{\rr {2d}} e^{-(|z|^2+|\zeta -\frac \kappa 2|^2 +|z-x|^2-|\zeta -\xi |^2)}
e^{-i(\scal z{\eta -\kappa}+\scal y\zeta)}\, dz d\zeta
\\[1ex]
=
2^{- \frac {5d}{2}} \pi^{-\frac d 2}  e^{-(\frac {|x|^2}2+\frac 12{|\xi -\frac \kappa 2|^2}
+\frac 18{|\eta -\kappa |^2} + \frac {|y|^2}8)}
e^{-\frac i2(\scal x{\eta -\kappa} +\scal y{\xi +\frac \kappa 2})}. 
\end{multline*}
Hence
\begin{multline}\label{Eq:STFTa}
V_\Phi a (x,\xi ,\eta ,y) 
\\[1ex]
= 2^{- 2d} \pi^{-d}
\sum _{\kappa \in \Lambda} c (\kappa)
e^{-(\frac {|x|^2}2+\frac 12{|\xi -\frac \kappa 2|^2}
+\frac 18{|\eta -\kappa |^2} + \frac {|y|^2}8)}
e^{-\frac i2(\scal x{\eta -\kappa} +\scal y{\xi +\frac \kappa 2})}.
\end{multline}

\par

If $q< \infty$ we get, in the third inequality using $p \le 1$, 
\begin{multline*}
\nm {W_{f,\phi}} {\splM^{p,q}}
= \nm a{\splM^{p,q}}\asymp \nm {\{ V_\Phi a (k_1,\kappa _1,\kappa _2,k_2)\}
_{k_j,\kappa _j\in \Lambda}}{\ell ^{p,q}(\Lambda ^4)}
\\[1ex]
\lesssim
\left ( \sum _{k_2,\kappa _2} \left ( \sum _{k_1,\kappa _1}\left (
\sum _\kappa c (\kappa) \, e^{-(\frac {|k_1|^2}2+\frac 12{|\kappa _1 -\frac \kappa 2|^2}
+\frac 18{|\kappa _2 -\kappa |^2} + \frac {|k_2|^2}8)}
\right )^p  \right )^{\frac qp}\right ) ^{\frac 1q}
\\[1ex]
\asymp
\left ( \sum _{\kappa _2} \left ( \sum _{\kappa _1}\left (
\sum _\kappa c(\kappa) \, e^{-(\frac 12{|\kappa _1 -\frac \kappa 2|^2}
+\frac 18{|\kappa _2 -\kappa |^2})}
\right )^p\right )^{\frac qp}\right ) ^{\frac 1q}
\\[1ex]
\le
\left ( \sum _{\kappa _2} \left (
\sum _{\kappa ,\kappa _1} c(\kappa)^p \, e^{-(\frac p2{|\kappa _1 -\frac \kappa 2|^2}
+\frac p8{|\kappa _2 -\kappa |^2})}
\right )^{\frac qp}\right ) ^{\frac 1q}
\\[1ex]
\asymp
\left ( \sum _{\kappa _2} \left (
\sum _{\kappa } c(\kappa)^p \, e^{-\frac p8{|\kappa _2 -\kappa |^2}}
\right )^{\frac qp}\right ) ^{\frac 1q}
=
\left (
\nm { \{c^p \} * e^{-\frac p8|\cdo |^2}}{\ell ^{\frac qp}}
\right )^{\frac 1p}
\\[1ex]
\le
\left (
\nm  {\{c^p \} }{\ell ^{\frac qp}} \nm {e^{-\frac p8|\cdo |^2}}{\ell ^1}
\right )^{\frac 1p}
\asymp
\nm  {\mathbf c }{\ell ^q} <\infty ,
\end{multline*}
using Young's inequality. The result follows if $q< \infty$.
If $q=\infty$ a similar argument proves the result.
\end{proof}

The preceding lemma is needed in the proof of Theorem
\ref{Thm:SharpnessBiLinCase} below on necessary conditions
for continuity. We aim at conditions on the exponents $p_j$,
$q_j$, $j=0,1,2$, that are necessary for 
\begin{equation}\label{Eq:CompContRel0}
\nm {a \wpr _A b}{M^{p_0,q_0}_{(\omega _0)}} \lesssim
\nm {a}{M^{p_1,q_1}_{(\omega _1)}}\nm {b}{M^{p_2,q_2}_{(\omega _2)}}
\end{equation}
to hold for all $a, b \in \mathscr S ({\rr {2d}})$, for certain weight
functions $\omega_j$, $j=0,1,2$.  We restrict to weights of
polynomial type. 

\par

By \cite[Proposition~2.8]{To14} it suffices to prove the result in
the Weyl case $A=1/2$,
and then \eqref{Eq:CompContRel0} in terms of symplectic
modulation spaces is 
\begin{equation}\label{Eq:CompContRel1}
\nm {a \wpr b}{\splM^{p_0,q_0}_{(\omega _0)}} \lesssim
\nm {a}{\splM^{p_1,q_1}_{(\omega _1)}}\nm {b}{\splM^{p_2,q_2}_{(\omega _2)}},
\quad a, b \in \mathscr S ({\rr {2d}}).
\end{equation}
The conditions on the weights \eqref{Eq:WeightCondBilinCase} and
\eqref{Ttdef2} are then transformed into
\begin{equation}\label{Eq:Weylweights}
\omega_0(Z+X,Z-X) \lesssim \omega_1(Y+X,Y-X) \,
\omega_2(Z+Y,Z-Y), \quad X,Y,Z \in \rr {2d}. 
\end{equation}
(Cf. \cite{Cordero1,Holst1}.)

\par

We will consider weights with the particular structure
\begin{equation}\label{weightcond2}
\begin{aligned}
\omega _0(X,Y) &= \frac {\vartheta _2(X-Y)}{\vartheta _0(X+Y)},\quad
\omega _1(X,Y)   = \frac {\vartheta _2(X-Y)}{\vartheta _1(X+Y)},
\\[1ex]
\omega _2(X,Y) &= \frac {\vartheta _1(X-Y)}{\vartheta _0(X+Y)},\quad
\end{aligned}
\end{equation}
for $\vartheta _j\in \mascP (\rr {2d})$, $j=0,1,2$. Then
\eqref{Eq:Weylweights} is automatically satisfied. 
Without loss we may assume $\vartheta _j \in C^\infty$
\cite[Remark~2.18]{Holst1}. 

\par

For $\vartheta \in  \mascP (\rr {2d})$ let $S^{(\vartheta)}(\rr {2d})$ denote
the space of smooth symbols on $\rr {2d}$ such that 
$(\partial^\alpha a)/\vartheta \in L^\infty$ for any $\alpha \in \nn{2d}$. 

\par

\begin{lemma}\label{lemweightreduction}
Let $p,q \in (0,\infty]$, let $\vartheta _j\in \mascP (\rr {2d})$, $j=1,2$ and suppose 
$\omega(X,Y) = \vartheta_2(X-Y)/\vartheta_1(X+Y)$. 
Then there exist $a_j \in S^{(\vartheta_j)}(\rr {2d})$ and
$b_j \in S^{(1/\vartheta_j)}(\rr {2d})$, $j=1,2$ such that 
\begin{equation}\label{eq:weylprod1}
a_j \wpr b_j = b_j \wpr a_j = 1, \quad j=1,2,  
\end{equation}
and the map $a \mapsto a_2 \wpr a \wpr b_1$ is continuous on
$\mathscr S(\rr {2d})$ and extends uniquely to a homeomorphism from 
$\splM^{p,q}_{(\omega)}(\rr {2d})$ to $\splM^{p,q}(\rr {2d})$. 
\end{lemma}

\par

\begin{proof}
According to \cite[Corollary~6.6]{Bony1} there exist
$a_j \in S^{(\vartheta_j)}(\rr {2d})$ and $b_j \in S^{(1/\vartheta_j)}(\rr {2d})$, $j=1,2$, 
such that \eqref{eq:weylprod1} is satisfied. 

\par

By \cite[Remark~2.18]{Holst1} we have
\begin{equation*}
S^{(\vartheta)}(\rr {2d}) = \bigcap_{N \ge 0}
\splM_{(1/\vartheta _N)}^{\infty,r}(\rr {2d}),\qquad
\vartheta _N(X,Y) = \vartheta(X) \eabs Y^{-N},
\end{equation*}
for any $\vartheta \in \mascP (\rr {2d})$ and any $r>0$ . More precisely
the remark gives the equality for $r=1$, and for general $r>0$,
the equality follows from the embeddings
\begin{equation*}
M^{\infty ,r_2}_{(1/v_{N+N_0})}\subseteq M^{\infty ,r_1}_{(1/v_N)}
\subseteq M^{\infty ,r_2}_{(1/v_N)},
\quad \text{when}\quad
r_1<r_2,\ N_0> 2d \left( \frac{1}{r_1} - \frac{1}{r_2}\right). 
\end{equation*}

\par

If we set $r = \min(1,p,q)$ then $p_1=\infty$, $q_1=r$, $p_2=p$, $q_2=q$, 
as well as $p_2=\infty$, $q_2=r$, $p_1=p$, $q_1=q$, satisfy the conditions 
\eqref{Eq:CaseLesspjp}, and \eqref{EqCaseqLessp} or \eqref{EqCasepLessq}
of Theorem \ref{Thm:BilinCase}. 

\par

From these observations the result follows from Theorem \ref{Thm:BilinCase}
and a repetition of the arguments in  the proof of \cite[Lemma~3.3]{Cordero1}. 
\end{proof}

\par

\begin{thm}\label{Thm:SharpnessBiLinCase}
Let $p_j,q_j\in (0,\infty ]$, suppose $\omega _j\in \mascP (\rr {4d})$, $j=0,1,2$,
are given by \eqref{weightcond2} where $\vartheta _j\in \mascP (\rr {2d})$, $j=0,1,2$. 
If \eqref{Eq:CompContRel1} holds then
\begin{equation}\label{Eq:LebExpNec}
\frac 1{p_0}\le \frac 1{p_1}+\frac 1{p_2},\quad
\frac 1{p_0}\le \frac 1{q_1}+\frac 1{q_2}
\quad \text{and}\quad
q_1,q_2\le q_0.
\end{equation}
\end{thm}

\par

\begin{proof}
By Lemma \ref{lemweightreduction} the estimate \eqref{Eq:CompContRel1} 
with weights \eqref{weightcond2} implies
\begin{equation}\label{Eq:CompContRel2}
\nm {a \wpr b}{\splM^{p_0,q_0}} \lesssim
\nm {a}{\splM^{p_1,q_1}} \nm {b}{\splM^{p_2,q_2}}, \quad
a\in \splM^{p_1,q_1}(\rr {2d}),\  b \in \mathscr S(\rr {2d}).
\end{equation}
It thus suffices to prove the result for $\omega _j \equiv 1$, $j=0,1,2$. 

\par

Let $a_{\lambda ,\mu} (x,\xi ) = e^{-\lambda |x|^2-\mu |\xi |^2}$ and
$a_\lambda = a_{\lambda ,\lambda}$, for 
$\mu ,\lambda >0$. 
Then by the proof of \cite[Proposition~3.1]{Holst1} (cf. \cite[Section~3]{Cordero1})
\begin{equation*}
\| a_\lambda \|_{\splM^{p,q}}^{1/d} 
= \pi ^{\frac 1p+\frac 1q - 1} p^{-\frac 1p}
q^{-\frac 1q}\lambda ^{-\frac 1p}(1+\lambda )^{\frac 1p+\frac 1q-1}
\end{equation*}
and
\begin{equation*}
a_\lambda \wpr a_\mu (X)
= (1+\lambda \mu)^{-d} \exp\left(-|X|^2 \frac{\lambda+\mu}{1+\lambda
\mu}\right).
\end{equation*}
Hence
\begin{equation*}
\begin{split}
\| a_\lambda \wpr a_\lambda \|_{\splM^{p,q}}^{1/d}
& = \pi^{1/p+1/q-1} p^{-1/p} q^{-1/q} 
\\[1ex]
& \qquad \times (1+\lambda^2)^{-1/q} (2 \lambda)^{-1/p} (1+\lambda)^{2(1/p+1/q-1)}. 
\end{split}
\end{equation*}
Thus
\begin{multline*}
\left( \frac{\| a_\lambda \wpr a_\lambda \|_{\splM^{p_0,q_0}}}
{\| a_\lambda  \|_{\splM^{p_1,q_1}} \| a_\lambda  \|_{\splM^{p_2,q_2}}}
\right)^{1/d}
\\[1ex] 
= C \lambda^{\frac{1}{p_1}+\frac{1}{p_2}-\frac{1}{p_0}}
(1+\lambda^2)^{-\frac{1}{q_0}} (1+\lambda)^{\frac{2}{p_0}
- \frac{1}{p_1} - \frac{1}{p_2} + \frac{2}{q_0} - \frac{1}{q_1}
- \frac{1}{q_2}}. 
\end{multline*}
for some constant $C>0$ which does not depend on $\lambda$.
The right hand side behaves like $\lambda^{\frac{1}{p_0}-\frac{1}{q_1}-\frac{1}{q_2}}$
when $\lambda$ is large, 
and like $\lambda^{\frac{1}{p_1}+\frac{1}{p_2}-\frac{1}{p_0}} $ when $\lambda$ is small. 
The continuity \eqref{Eq:CompContRel2} hence implies the necessary conditions 
\begin{equation*}
\frac 1{p_0}\le \frac 1{q_1}+\frac 1{q_2}, \quad \frac 1{p_0}\le \frac 1{p_1}+\frac 1{p_2}. 
\end{equation*}

It remains to show $q_1,q_2\le q_0$. 
Since $\overline{a_1 \wpr a_2} = \overline{a}_2 \wpr \overline{a}_1$ (cf. \cite{Holst1}), 
it suffices to show $q_1 \le q_0$. 
We give a proof by contradiction. Suppose \eqref{Eq:CompContRel2} holds
and $q_0 < q_1$.
Let $\Lambda \subseteq \rr d$ be a lattice,
\begin{align*}
\phi(x) &= \pi ^{-d/4} e ^{-|x| ^2 /2},
\qquad
\bold c = \{c (\kappa )\} _{\kappa \in \Lambda} \in \ell^{q _1}(\Lambda)
\setminus  \ell^{q _0}(\Lambda),
\intertext{and let}
f(x) &= \sum _{\kappa \in \Lambda}  c (\kappa ) e ^{i \eabs {x, \kappa}} \phi (x).
\end{align*}
Then
\begin{gather*}
f \in \bigcap _{p _1 > 0} M ^{p _1, q _1}(\rr {d}) \setminus M ^{\infty, q_0}(\rr {d}),
\\[1ex]
a =W _{f, \phi} \in \bigcap _{p _1 > 0}  \splM ^{p _1, q _1}(\rr {2d})
\quad \text{and} \quad 
b =W _{\phi, \phi} \in \mascS (\rr {2d}),
\end{gather*}
by Lemma \ref{lemFseries}. Since
$$
\op ^w (a) g = (2 \pi )^{-d /2}(g, \phi) f
\quad \text{and}\quad
\op ^w (b) g = (2 \pi )^{-d /2}(g, \phi) \phi ,
$$
it follows that
\begin{equation*}
\op ^w (a \# b) \phi = (2 \pi) ^{-d} \|\phi\| _{L ^2} ^4 f
\in \bigcap _{p _1 > 0} M ^{p _1, q _1}(\rr {d}) \setminus
 M ^{\infty, q_0}(\rr {d}).
\end{equation*}
Therefore $\op ^w (a \# b)$ is not continuous from $\mascS (\rr d)$ 
to $M ^{\infty, q_0}(\rr d)$. 

\par

On the other hand we have by assumption 
\begin{equation*}
a \# b \in \splM ^{p _0, q_0} \subseteq \splM ^{\infty, q_0}.
\end{equation*}

\par

If $q _0 \in (0, 1]$, then $\op ^w (a \# b)$ is continuous
from $M ^{p _0, q_0}$ to $M ^{p _0, q_0}$ when $p _0 \in [q _0, \infty]$,
by \cite [Theorem 3.1]{To13}. This contradicts the fact that  
$\op ^w (a \# b)$ is not continuous from $\mascS$ 
to $ M ^{\infty, q_0}$. Hence the assumption $q_0 < q_1$
must be false.

\par

If instead $q _0 \in [1, \infty]$, then by \cite [Theorem 4.3]{Toft2}
$\op ^w (a \# b)$ is continuous from $M ^{1, 1}$
to $M ^{q _0, q_0}$, which
again contradicts the fact that
$\op ^w (a \# b)$ is not continuous from $\mascS$ 
to $ M ^{\infty, q_0}$.
Hence the assumption $q_0 < q_1$ is again false. 

Thus we must have $q_1 \le q_0$. 
\end{proof}

\begin{rem}
Let $\mathscr P_E^0(\rr d)$ denote all $\omega \in \mathscr P_E(\rr d)$
such that $\omega$ is $v$-moderate for a submultiplicative weight $v$ satisfying
\begin{equation*}
v(x) \lesssim e^{r|x|}, \quad x \in \rr d, 
\end{equation*}
for all $r>0$. 
Then $\mathscr P(\rr d) \subsetneq \mathscr P_E^0(\rr d)$. 
By using the new \cite[Theorem~4.1]{Abdeljawad1} instead of Lemma \ref{lemweightreduction} it follows that Theorem \ref{Thm:SharpnessBiLinCase}
holds for $\vartheta_j \in \mathscr P_E^0(\rr {2d})$ and $\omega_j$ defined by \eqref{weightcond2}.
The space $\mascS$ in \eqref{Eq:CompContRel1} is then replaced by $\maclS_1$. 
\end{rem}

\begin{rem}
For Banach modulation spaces with exponents $p_j, q_j$ restricted to $[1,\infty]$ we have found that the following conditions are necessary and sufficient for continuity of the Weyl product \cite[Theorems 0.1 and 3.1]{Cordero1}. 
\begin{align}
\frac1{p_0} & \le \frac1{p_1} + \frac1{p_2},
\label{suff1}
\\[1ex]
q_1, q_2 & \le q_0, \quad 
1 \le \frac1{q_1}  +  \frac1{q_2}, 
\label{suff3}
\\[1ex]
\frac1{p_0} + \frac1{q_0} & \leq \frac1{q_1} + \frac1{q_2}, \quad
1 + \frac1{q_0}  \leq \frac1{q_1} + \frac1{q_2} + \frac1{p_j}, \quad j=1,2, 
\label{suff4}
\\[1ex]
1 +  \frac1{p_0} + \frac1{q_0} & \leq \frac1{q_1} + \frac1{q_2} + \frac1{p_1} + \frac1{p_2}. 
\label{suff5}
\end{align}

In this paper we have worked with exponents $p_j, q_j$ in the full range $(0,\infty]$. 
The sufficient conditions in Theorem \ref{Thm:BilinCase}
and the necessary conditions in Theorem \ref{Thm:SharpnessBiLinCase} are not equal, as conditions \eqref{suff1} -- \eqref{suff5} are for 
exponents in $[1,\infty]$. 

In fact, consider the inclusion 
\begin{equation*}
M^{1,2} \wpr M^{1,2} \subseteq M^{\infty,2} 
\end{equation*}
which holds since the exponents satisfy \eqref{suff1} -- \eqref{suff5}. 
They do however not satisfy \eqref{Eq:CaseLesspjp}, and \eqref{EqCaseqLessp} or \eqref{EqCasepLessq}. 
Hence the sufficient conditions in Theorem \ref{Thm:BilinCase} are not all necessary. 
\end{rem}

%%%%%%%%%%%%%%%%%%%%%
\section*{Appendix}\label{sec:appendix}
%%%%%%%%%%%%%%%%%%%%%

In this appendix we prove the formula 
\begin{equation}\tag*{(A.1)}
\mathscr F(f\, \overline{\phi})(\xi ) =
(2\pi )^{-d/2} (f,\phi \, e^{i\scal \cdo \xi}), \quad f \in \mathcal S_{s}'(\rr d), \quad
\phi \in \mathcal S_{s}(\rr d), \quad \xi \in \rr d, 
\end{equation}
for $s \geq 1/2$,  which we claimed to be true in the definition of the
STFT \eqref{defstft}.  There is a parallel formula for $f \in \Sigma _{s}'
(\rr d)$, $\phi \in \Sigma_{s}(\rr d)$ and $s> 1/2$, that we also prove. 

\par

Let $f \in \mathcal S_{s}'(\rr d)$, let $\phi \in \mathcal S_{s}(\rr d)$ and denote 
\begin{equation}\tag*{(A.2)}
u(\xi ) =
(2\pi )^{-d/2} (f,\phi \, e^{i\scal \cdo \xi}) = (2\pi )^{-d/2} \langle f,
\overline{\phi} \, e^{-i\scal \cdo \xi} \rangle, \quad \xi \in \rr d. 
\end{equation}
Then $u \in C^\infty(\rr d)$. We need the following estimate (cf. \cite{ChuChuKim}). 

\renewcommand{\rubrik}{Lemma A.1}

\begin{tom} 
The function {\rm (A.2)} satisfies the estimate
\begin{equation*}
%\tag*{(A.3)}
| u (\xi)| \lesssim e^{c |\xi|^{1/s}}, \quad \xi \in \rr d,  
\end{equation*}
for any $c>0$. 
\end{tom}

\par

\begin{proof}
By \eqref{GSspacecond1} $\phi \in \mathcal S_{s,h}(\rr d)$ for all $h \geq h_0$ where $h_0 > 0$. 
Let $c>0$ and set 
\begin{equation*}
h = \max \left( h_0, \left( \frac{ds}{c} \right)^s \right). 
\end{equation*}

Let $\alpha,\beta \in \nn d$. 
Using $|\gamma|! \leq d^{|\gamma|} \gamma!$ (cf. \cite[Eq.~(0.3.3)]{Nicola1})
and $\sum_{\gamma \leq \beta} \binom{\beta}{\gamma} = 2^{|\beta|}$ we estimate for $x,\xi \in \rr d$
\begin{align*}
\frac{ \left|  x^\alpha \partial_x^\beta ( e^{i\scal x \xi} \phi(x) ) \right|}{(\alpha! \beta!)^s (2h)^{|\alpha+\beta|}}
& \leq 2^{-|\alpha+\beta|} \sum_{\gamma \leq \beta} \binom{\beta}{\gamma} \frac{|\xi|^{|\gamma|} \ h^{-|\gamma|}}{\gamma!^s} 
\frac{\left|  x^\alpha \partial^{\beta-\gamma} \phi(x) \right|}{(\alpha! (\beta-\gamma)!)^s h^{|\alpha+\beta-\gamma|}} \binom{\beta}{\gamma}^{-s} \\
& \leq \nm \phi {\mathcal S_{s,h}} 2^{-|\alpha+\beta|} \sum_{\gamma \leq \beta} \binom{\beta}{\gamma}
\left( \frac{ \left(\frac{c}{s}|\xi|^{1/s} \right)^{|\gamma|}}{|\gamma|!} \right)^s \left( \frac{(d s)^s}{h c^s} \right)^{|\gamma|} \\
& \lesssim 2^{-|\alpha+\beta|} \sum_{\gamma \leq \beta} \binom{\beta}{\gamma}
\left( \frac{ \left(\frac{c}{s}|\xi|^{1/s} \right)^{|\gamma|}}{|\gamma|!} \right)^s \\
& \leq e^{c |\xi|^{1/s}} 2^{-|\alpha+\beta|} \sum_{\gamma \leq \beta} \binom{\beta}{\gamma} \\
& \leq e^{c |\xi|^{1/s}}. 
\end{align*}

This implies
\begin{equation*}
\nm {\phi \, e^{i\scal \cdo \xi}} {\mathcal S_{s,2h}} \lesssim e^{c |\xi|^{1/s}}, \quad \xi \in \rr d, 
\end{equation*}
which via \eqref{GSspacecond1}$'$ finally gives the estimate
\begin{equation*}
| u (\xi)| \lesssim \nm {\phi \, e^{i\scal \cdo \xi}} {\mathcal S_{s,2h}} \lesssim e^{c |\xi|^{1/s}}, \quad \xi \in \rr d.   
\end{equation*}
\end{proof}

\par
 
The formula (A.1) amounts to the claim $\mathscr F(f\, \overline{\phi}) = u$. 

\par

A priori $\mathscr F(f\, \overline{\phi}) \in \mathcal S_{s}'(\rr d)$ is the distribution 
\begin{equation*}
\langle \mathscr F(f\, \overline{\phi}), g \rangle = \langle f, \overline{\phi}
\, \widehat g \rangle, \quad g \in \mathcal S_{s}(\rr d). 
\end{equation*}
To prove our claim $\mathscr F(f\, \overline{\phi}) = u$ we must therefore show 
\begin{equation}\tag*{(A.3)}
\langle f, \overline{\phi} \, \widehat g \rangle = \int_{\rr d} u(x) \, g(x) \, dx
= (2\pi )^{-d/2} \int_{\rr d} \langle f, \overline{\phi} \, g(x) \, e^{-i\scal
\cdo x} \rangle \, dx, \quad g \in \mathcal S_{s}(\rr d).
\end{equation}

\par

Note that the integral is well defined due to Lemma A.1 and the estimate
for $g \in \maclS _{s}(\rr d)$ \cite[Lemma~1.6]{Toft8}
\begin{equation*}
|g(x)| \lesssim e^{-\ep |x|^{1/s}}, \quad x \in \rr d, 
\end{equation*}
which is valid for some $\ep>0$. 

\par

In view of the definition of the Fourier transform $\widehat g$,
formula (A.3) is true provided we can switch order in
the action of the distribution $1 \otimes f \in \maclS _{s}'(\rr {2d})$
with respect to the first and second $\rr d$ variable, 
when it acts on the test function $\Phi(x,y) = (2\pi )^{-d/2}
\overline{\phi}(y) \, g(x) e^{-i\scal y x}$.  Note that $\Phi \in \maclS _{s}
(\rr {2d})$ if $\phi,g \in \mathcal S_{s} (\rr d)$ and $s \geq 1/2$, and
$\Phi \in \Sigma_{s} (\rr {2d})$ if $\phi,g \in \Sigma_{s}(\rr d)$ and
$s > 1/2$, cf. \cite[Theorem~3.1]{CaTo} and \cite[Proposition~3.4]{Carypis1}.

\par

Thus the claim (A.1) is a consequence of the following Fubini-type result for
Gelfand--Shilov distributions. It corresponds to
\cite[Theorem~5.1.1]{Hormander1} in the Schwartz distribution theory. 

\par

\renewcommand{\rubrik}{Theorem A.2}

\begin{tom} 
Suppose $s \geq 1/2$, and  $f_j \in \maclS _{s}'(\rr {d_j})$, $j=1,2$. 
Then there exists a unique tensor product distribution $f = f_1
\otimes f_2 \in \maclS _{s}'(\rr {d_1+d_2})$ such that 
\begin{equation*}
\langle f_1 \otimes f_2, \phi_1 \otimes \phi_2 \rangle =
\langle f_1, \phi_1 \rangle \langle f_2, \phi_2 \rangle,
\quad \phi_j \in \mathcal S_{s}(\rr {d_j}), \quad j=1,2. 
\end{equation*}
It holds
\begin{equation*}
\langle f, \phi \rangle 
= \langle f_1, \langle f_2, \phi(x_1,x_2) \rangle \rangle
= \langle f_2, \langle f_1, \phi(x_1,x_2 )\rangle \rangle,
\quad \phi \in \mathcal S_{s}(\rr {d_1+d_2}), 
\end{equation*}
where $f_j$ acts on $x_j$ only, $j=1,2$. 

\par

The same conclusion holds for $s > 1/2$ and  $f_j \in
\Sigma_{s}'(\rr {d_j})$, $j=1,2$, with test functions in $\Sigma_s$. 
\end{tom}

\par

\begin{proof}
We use the Hermite functions 
\begin{equation*}
h_\alpha (x) = \pi ^{-\frac d4}(-1)^{|\alpha |}
(2^{|\alpha |}\alpha !)^{-\frac 12}e^{\frac {|x|^2}2}
(\partial ^\alpha e^{-|x|^2}), \quad x \in \rr d, \quad \alpha \in \nn d,  
\end{equation*}
and formal series expansions with respect to Hermite functions: 
\begin{equation*}
f = \sum_{\alpha \in \nn d} c_\alpha h_\alpha
\end{equation*}
where $\{ c_\alpha \}$ is a sequence of complex coefficients
defined by $c_\alpha = c_\alpha(f) = (f,h_\alpha)$. 

\par

It is known that Gelfand--Shilov spaces and their distribution
duals can be identified by means of such series expansions, 
with characterizations in terms of the corresponding sequence
spaces (see \cite{GrPiRo2,GraLecPilRod,Toft15} and the references
therein). 

\par

In fact, let
\begin{equation*}
f = \sum_{\alpha \in \nn d} c_\alpha h_\alpha 
\end{equation*}
and 
\begin{equation*}
\phi = \sum_{\alpha \in \nn d} d_\alpha h_\alpha 
\end{equation*}
with sequences $\{ c_\alpha\}$ and $\{d_\alpha \}$ of finite support.
Then the sesquilinear form
\begin{equation}\tag*{(A.4)}
( f, \phi ) = \sum_{\alpha \in \nn d} c_\alpha \overline{d_\alpha} 
\end{equation}
agrees with the inner product on $L^2(\rr d)$ due to the fact that
$\{ h_\alpha \}_{\alpha \in \nn d} \subseteq L^2(\rr d)$ is an orthonormal basis. 
The form (A.4) extends uniquely to the duality on $\maclS
_{s}'(\rr d) \times \mathcal S_{s}(\rr d)$ for $s \geq 1/2$, 
and to the duality on $\Sigma_{s}'(\rr d) \times \Sigma_{s}(\rr d)$ for $s > 1/2$.
All spaces are then expressed in terms of the Hilbert sequence spaces   
\begin{equation*}
\ell_{r}^2 = \ell_{r}^2(\nn d)  = \Sets{ \{c_\alpha\}}
{\sum_{\alpha \in \nn d} |c_\alpha|^2 e^{r |\alpha|^{\frac{1}{2s}}} < \infty}
\end{equation*}
where $r \in \re$. 
For $s \geq 1/2$ the space $\maclS _{s}(\rr d)$
is identified topologically as the inductive limit 
\begin{equation*}
\mathcal S_{s}(\rr d) = \bigcup_{r>0} \Sets{ \sum_{\alpha \in \nn d}
c_\alpha h_\alpha}{ \{c_\alpha\} \in \ell_{r}^2 }
\end{equation*}
and $\maclS _{s}'(\rr d)$ is identified topologically as the projective limit 
\begin{equation*}
\maclS _{s}'(\rr d) = \bigcap_{r>0}
\Sets{ \sum_{\alpha \in \nn d} c_\alpha h_\alpha}{\{c_\alpha\} \in \ell_{-r}^2 }.
\end{equation*}

For $s > 1/2$ the space $\Sigma_{s}(\rr d)$ is identified topologically
as the projective limit 
\begin{equation*}
\Sigma_{s}(\rr d) = \bigcap_{r>0} \left\{ \sum_{\alpha \in \nn d}
c_\alpha h_\alpha : \ \{c_\alpha\} \in \ell_{r}^2 \right\}
\end{equation*}
and $\Sigma_{s}'(\rr d)$ is identified topologically as the inductive limit 
\begin{equation*}
\Sigma_{s}'(\rr d) = \bigcup_{r>0} \Sets {\sum_{\alpha \in \nn d} c_\alpha h_\alpha}
{\{c_\alpha\} \in \ell_{-r}^2} .
\end{equation*}

\par

We have for $\alpha = (\alpha_1,\alpha_2) \in \nn {d_1+d_2}$ with
$\alpha_j \in \nn {d_j}$, $j=1,2$, $h_\alpha = h_{\alpha_1} \otimes h_{\alpha_2}$. 
This gives for $f_j \in \mathcal S_{s}'(\rr {d_j})$, $j=1,2$,
\begin{equation*}
c_\alpha = c_\alpha(f_1 \otimes f_2) = (f_1, h_{\alpha_1}) (f_2, h_{\alpha_2}), \quad \alpha = (\alpha_1,\alpha_2) \in \nn {d_1+d_2},
\end{equation*}
so $c_\alpha = c_{1,\alpha_1} c_{2,\alpha_2}$ if we denote
$c_{j,\alpha_j} = (f_j, h_{\alpha_j})$ where $\alpha_j \in \nn {d_j}$ for $j=1,2$.

\par

Let $\phi \in \mathcal S_{s}(\rr {d_1+d_2})$ and denote
$d_\alpha (\phi)= (\phi,h_\alpha)$ for $\alpha \in \nn {d_1+d_2}$. 
This gives for any $r>0$
\begin{equation}\tag*{(A.5)}
\langle f_1 \otimes f_2, \phi \rangle = 
\sum_{(\alpha_1,\alpha_2) \in \nn {d_1+d_2}} c_{1,\alpha_1} c_{2,\alpha_2}
\, e^{-r |(\alpha_1,\alpha_2)|^{\frac{1}{2s}}} \,
\overline{d_{\alpha_1,\alpha_2}} \, e^{r |(\alpha_1,\alpha_2)|^{\frac{1}{2s}}} .
\end{equation}
From
\begin{equation*}
e^{-r |(\alpha_1,\alpha_2)|^{\frac{1}{2s}}}
\le e^{-\frac{r}{2} |\alpha_1|^{\frac{1}{2s}}}
e^{-\frac{r}{2} |\alpha_2|^{\frac{1}{2s}}}, 
\end{equation*}
$\{ c_{1,\alpha_1} \} \in \ell_{-r}^2(\nn {d_1})$, $\{ c_{2,\alpha_2} \}
\in \ell_{-r}^2(\nn {d_2})$ for any $r>0$,  $\{ d_{\alpha_1,\alpha_2} \}
\in \ell_{r}^2(\nn {d_1+d_2})$ for some $r>0$, and the Cauchy--Schwarz inequality, 
we may now conclude that the sum (A.5) converges absolutely. 

\par

The conclusion of the theorem is thus a consequence of the well known
Fubini theorem with respect to the counting measure. 
\end{proof}

\par

\end{document}